\documentclass{amsart}
\usepackage{amsmath, amssymb, mathtools}
\usepackage{accents}
\usepackage{graphicx}

\newtheorem{thm}{Theorem}[section]

\newtheorem{lem}[thm]{Lemma}

\theoremstyle{definition}

\theoremstyle{remark}
\newtheorem{rem}[thm]{Remark}
\numberwithin{equation}{section}

\newcommand{\bbR}{\mathbb{R}}

\newcommand{\va}{\mathbf{a}}
\newcommand{\vb}{\mathbf{b}}
\newcommand{\vc}{\mathbf{c}}
\newcommand{\vn}{\mathbf{n}}
\newcommand{\vp}{\mathbf{p}}
\newcommand{\vq}{\mathbf{q}}
\newcommand{\vt}{\mathbf{t}}
\newcommand{\vu}{\mathbf{u}}
\newcommand{\vv}{\mathbf{v}}
\newcommand{\vx}{\mathbf{x}}

\newcommand{\vtau}{\boldsymbol{\tau}}

\newcommand{\tvq}{\tilde{\vq}}
\newcommand{\tW}{\tilde{W}}

\renewcommand{\div}{\mathrm{div}}
\newcommand{\curl}{\mathrm{curl}}
\newcommand{\grad}{\mathrm{grad}}

\newcommand{\al}{\hspace*{0.6cm}}

\newcommand{\cE}{\mathcal{E}}
\newcommand{\cF}{\mathcal{F}}
\newcommand{\cH}{\mathcal{H}}
\newcommand{\cM}{\mathcal{M}}
\newcommand{\cN}{\mathcal{N}}
\newcommand{\cP}{\mathcal{P}}
\newcommand{\cT}{\mathcal{T}}
\newcommand{\cV}{\mathcal{V}}
\begin{document}

\title[$H(\curl)$ and $H(\div)$ elements on polytopal meshes]{Minimal degree $H(\curl)$ and $H(\div)$ conforming finite elements on polytopal meshes}%
\author{Wenbin Chen}%
\address{Department of Mathematics, Fudan University, Shanghai, China}%
\email{wbchen@fudan.edu.cn}%
\author{Yanqiu Wang} %
\address{Department of Mathematics, Oklahoma State University, Stillwater, OK, USA} %
\email{yanqiu.wang@okstate.edu} %


\begin{abstract}
We construct $H(\curl)$ and $H(\div)$ conforming finite elements on
convex polygons and polyhedra with minimal possible degrees of
freedom, i.e., the number of degrees of freedom is equal to the
number of edges or faces of the polygon/polyhedron. The construction
is based on generalized barycentric coordinates and the Whitney
forms. In 3D, it currently requires the faces of the polyhedron be
either triangles or parallelograms. Formula for computing basis
functions are given. The finite elements satisfy discrete de Rham
sequences in analogy to the well-known ones on simplices. Moreover,
they reproduce existing $H(\curl)$-$H(\div)$ elements on simplices,
parallelograms, parallelepipeds, pyramids and triangular prisms.
Approximation property of the constructed elements is also analyzed,
by showing that the lowest-order simplicial
N\'{e}l\'{e}lec-Raviart-Thomas elements are subsets of the
constructed elements on arbitrary polygons and certain polyhedra.
\end{abstract}

\maketitle

\section{Introduction}
On a contractible smooth manifold $T\subset \bbR^m$, it is
well-known \cite{Arnold06, Arnold06b, Arnold09, Arnold10} that the
extended $L^2$ de Rham complex
\begin{equation} \label{eq:deRham}
0 \xrightarrow{\al} \bbR \xrightarrow[\al]{\subset} H\Lambda^0(T)
\xrightarrow[\al]{d} H\Lambda^1(T) \xrightarrow[\al]{d} \cdots
\xrightarrow[\al]{d} H\Lambda^m(T) \xrightarrow{\al} 0,
\end{equation}
is exact, where $d$ is the exterior derivative, and $H\Lambda^k(T)$,
$k=0,\ldots ,m$, are Hilbert spaces containing all differential
$k$-forms $\omega$, such that both $\omega$ and $d\omega$ are in
$L^2$. Using traditional vector proxy notation of differential
forms, the de Rham complex can be expressed in 3D as
$$
0 \xrightarrow{\al} \bbR \xrightarrow[\al]{\subset} H^1(T)
\xrightarrow[\al]{\grad} H(\curl,\,T) \xrightarrow[\al]{\curl}
H(\div,\,T) \xrightarrow[\al]{\div} L^2(T) \xrightarrow{\al} 0,
$$
and in 2D as either one of the following
$$
\begin{aligned}
0 \xrightarrow{\al} \bbR \xrightarrow[\al]{\subset} H^1(T)
\xrightarrow[\al]{\grad} H(\curl,\,T) \xrightarrow[\al]{\curl}
L^2(T) \xrightarrow{\al} 0, \\ 0 \xrightarrow{\al} \bbR
\xrightarrow[\al]{\subset} H^1(T) \xrightarrow[\al]{\curl}
\,H(\div,\,T)\, \xrightarrow[\al]{\div} L^2(T) \xrightarrow{\al} 0,
\end{aligned}
$$
where we conveniently denote the 2D $\curl$ operator by $\curl=
\begin{bmatrix}-\partial_y \\ \partial_x\end{bmatrix}$. Note that the two complexes in 2D are indeed equivalent under
the following mapping
$$
H(\curl,\,T) \xleftrightarrow[\;\chi^{-1}\cdot\;]{\chi\cdot}
H(\div,\,T),\qquad \textrm{where }\chi =
\begin{bmatrix}0&-1\\1&0\end{bmatrix}.
$$
Thus it suffices to only study one of them, and in this paper we
pick the one containing $H(\div,\,T)$.

The idea of finite element exterior calculus is to build finite
dimensional sub-complexes of (\ref{eq:deRham}), and then patch the
local discrete spaces on each mesh element, usually a polytope,
together to obtain the finite element space on the entire mesh. To
build conforming finite element spaces, certain continuity
conditions will be imposed on the boundary of $T$. When $T$ is a
simplex or a hypercube, it is well-known that such sub-complexes can
be built using polynomials, i.e., $\mathcal{P}_r\Lambda^k$,
$\mathcal{P}_r^{-}\Lambda^k$ and $\mathcal{H}_r\Lambda^k$, for $0\le
k\le m$ (see \cite{Arnold06} for definition of these spaces). Here
we are interested in more general polygonal/polyhedral domain $T$,
on which polynomial spaces like $\mathcal{P}_r\Lambda^k$,
$\mathcal{P}_r^{-}\Lambda^k$ and $\mathcal{H}_r\Lambda^k$ are
usually not enough for building conforming finite elements. For
example, in 2D, one can not build $H^1$-conforming, piecewise
linear/bilinear, scalar finite element space on meshes containing
$n$-gons with $n>4$. A solution is to use the generalized
barycentric coordinates: Wachspress, Sibson, harmonic, and mean
value, etc. (see \cite{Floater03, Floater05, Floater15, Joshi07,
Martin08, Sibon90, Wachspress75, Wachspress11, Warren96} and
references therein), which allows one to build $H^1$-conforming
scalar finite element spaces using a larger set of basis functions
\cite{Floater14, Gillette12, Manzini14, Rand13, Sukumar04,
Sukumar06, Talischi09, Talischi10, Wicke07}. For example, the
Wachspress element uses rational functions. We would also like to
mention two methods related to the generalized barycentric
coordinates: the mimetic finite difference method (see the recent
survey paper \cite{Lipnikov14}) and the virtual element method
\cite{Veiga13}. Both methods are defined on general polytopes. Among
them, the lowest order virtual element method is indeed equivalent
to an $H^1$ conforming finite element using a set of harmonic
barycentric coordinates.

Recall the traditional polynomial-valued barycentric coordinates
defined on simplices, generalized barycentric coordinates
$\{\lambda_i\}$, for $i$ from $1$ to the number of vertices, can be
viewed as extensions of traditional barycentric coordinates to a
polytope $T$. According to the construction, they may have some nice
properties, which will be further explained later. In general, we
expect $\{\lambda_i\}$  to form a basis for an $H^1$ conforming
scalar finite element on $T$. Extending such elements to $H(\curl)$
and $H(\div)$ on general polytopes is not easy. As early as in 1988,
researchers have realized the important role of Whitney forms in
constructing vector-valued finite element spaces \cite{Bossavit88}.
The Whitney $1$-form and Whitney $2$-form on simplices are defined,
respectively, by
\begin{align}
W_{ij} &= \lambda_i \nabla \lambda_j - \lambda_j \nabla \lambda_i,
\label{eq:Whitney1}
\\
W_{ijk} &= \lambda_i \nabla \lambda_j\times \nabla \lambda_k +
\lambda_j \nabla \lambda_k\times \nabla \lambda_i + \lambda_k \nabla
\lambda_i\times \nabla \lambda_j. \label{eq:Whitney2}
\end{align}
Formally, by using generalized barycentric coordinates, they can be
extended to general polytopes. There were several pioneering works
on extending the Whitney forms and building $H(\curl)$/$H(\div)$
conforming finite elements over non-simplicial polytopes, including
polygons \cite{Euler06}, rectangular grids \cite{Gradinaru02}, and
pyramids \cite{Hiptmair99}. In recent years, this idea has attracted
more attentions. Gillette and Bajaj \cite{Gillette10, Gillette11}
constructed dual mixed finite elements on polytopal meshs generated
by taking the dual of simplicial meshes. Later in \cite{Bossavit10},
Bossavit constructed edge-based and face-based Whitney forms on
tetrahedra, hexahedra, triangular prisms, and pyramids using
techniques called `conation' and `extrusion'. And in the most recent
work \cite{Gillette14}, Gillette, Rand and Bajaj constructed
$H(\curl)$ and $H(\div)$ conforming finite elements on arbitrary
polytopes using the span of all Whitney $1$-forms and $2$-forms,
respectively. We would also like to mention a few related works not
using the Whitney forms. Kuznetsov and Repin \cite{Kuznetsov04,
Kuznetsov08} constructed $H(\div)$ elements on polytopes with
simplicial refinements by solving a local discrete mixed problem.
Christiansen \cite{Christiansen08} constructed $H(\curl)$ and
$H(\div)$ conforming finite elements on polytopes by using harmonic
basis functions, which are known to be almost non-computable.
Klausen, Rasmussen and Stephansen \cite{Klausen12} directly
constructed $H(\div)$ conforming elements on polygons and simple
polyhedra using generalized barycentric coordinates. A polyhedron in
3D is simple if all its vertices are connected to exactly 3 edges.
The elements constructed in \cite{Klausen12}, although having
minimal degrees of freedom, does not fit easily into a de Rham
sequence.

The main purpose of this paper is to provide a unified,
easy-to-compute, and minimal degree construction of $H(\curl)$ and
$H(\div)$ conforming finite elements on convex polytopes, that
satisfy the discrete de Rham sequence. Let us briefly explain how
our work will be different from the existing results mentioned
above. We aim at building sub-complexes of (\ref{eq:deRham}) using
the minimal amount of basis functions that ensures $H(\curl)$ and
$H(\div)$ conformity. At the same time, we want the element to be
constructed provides at least $O(h)$ approximation rate. Let us
first recall the spaces constructed in \cite{Gillette14}. Define
$$
\mathcal{W}\Lambda^0(T) = span\{\lambda_i\},\quad
\mathcal{W}\Lambda^1(T) = span\{W_{ij}\},\quad
\mathcal{W}\Lambda^2(T) = span\{W_{ijk}\}.
$$
In \cite{Gillette14}, the authors have proved that the above defined
finite element spaces are $H^1$/$H(\curl)$/$H(\div)$ conforming and
contain $\mathcal{P}_1^{-}\Lambda^k(T)$, the lowest-order
N\'{e}d\'{e}lec-Raviart-Thomas spaces on simplices defined as
following:
\begin{equation} \label{eq:WLambdaTincludeP-LambdaT}
\begin{aligned}
\textrm{In 2D:}\qquad\quad \mathcal{W}\Lambda^0(T)  \supseteq\, &\cP_1^-\Lambda^0(T) = span\{1, x,  y\}, \\
\chi(\mathcal{W}\Lambda^1(T))  \supseteq\,  & \cP_1^-\Lambda^1(T) =
\{a\vx + \vc,\textrm{ for } a\in \bbR,\,
\vc\in\bbR^2\}, \\
\textrm{In 3D:}\qquad\quad \mathcal{W}\Lambda^0(T)  \supseteq\, &\cP_1^-\Lambda^0(T) = span\{1, x,  y, z\}, \\
\mathcal{W}\Lambda^1(T)  \supseteq\, & \cP_1^-\Lambda^1(T) =
\{\va\times\vx + \vb,\textrm{ for
}\va,\vb\in\bbR^3\},\\
\mathcal{W}\Lambda^2(T)  \supseteq\, & \cP_1^-\Lambda^2(T) = \{a\vx
+ \vc,\textrm{ for }a\in\bbR,\, \vc\in\bbR^3\}.
\end{aligned}
\end{equation}
Moreover, if $T$ is a simplex, then $\mathcal{W}\Lambda^k(T)$
coincides with $\mathcal{P}_1^{-}\Lambda^k(T)$, i.e., all
$\supseteq$ in the above become $=$.

Clearly, $\mathcal{W}\Lambda^0(T)$ is one of the smallest possible
scalar finite elements on $T$ that can ensure $H^1$ conformity.
However, $\mathcal{W}\Lambda^1(T)$/$\mathcal{W}\Lambda^2(T)$ are far
from the smallest $H(\curl)$/$H(\div)$ conforming elements on
general polytopes. Indeed, denote by $n$ the total number of
vertices in $T$, then one has
$$
\begin{aligned}
\textrm{total number of }W_{ij} &= \begin{pmatrix}n \\
2\end{pmatrix}, \\
\textrm{total number of }W_{ijk} &= \begin{pmatrix}n \\
3\end{pmatrix}.
\end{aligned}
$$
For example, when $T$ is a 3D cube, the above two numbers are $28$
and $56$, respectively. It is not clear whether $W_{ij}$ (or
$W_{ijk}$) are linearly independent or not. Thus one may need to use
the least squares method in the implementation. Comparing to the
known smallest vector-valued finite element complex on a cube
\cite{Nedelec80}, which uses $12$ basis functions in the $H(curl)$
element and $6$ basis functions in the $H(\div)$ element, the spaces
$\mathcal{W}\Lambda^1(T)$ and $\mathcal{W}\Lambda^2(T)$ may contain
too much redundant information.

We want to find the minimal discrete de Rham complex on general
convex polytopes that provides conforming approximations in $H^1$,
$H(\curl)$ and $H(\div)$. Because of the nice property of Whitney
forms \cite{Bossavit88, Whitney57}, we limit our searching in
subsets of $\mathcal{W}\Lambda^k(T)$. That is, we shall construct
finite elements $\mathcal{M}\Lambda^k(T)$ satisfying
$$
\mathcal{M}\Lambda^0(T) =
\mathcal{W}\Lambda^0(T)\qquad\textrm{and}\qquad
\mathcal{M}\Lambda^k(T) \subseteq
\mathcal{W}\Lambda^k(T)\quad\textrm{for }k=1,2.
$$
Now let us look at the smallest possible dimension of
$\mathcal{M}\Lambda^k(T)$, for $k=1,2$, on convex polytopes. We
start from the 3D case. Denote by $\#V$, $\#E$ and $\#F$ the number
of vertices, edges and faces of a convex polyhedron $T$. Then, one
has $dim \mathcal{M}\Lambda^0(T) = dim \mathcal{W}\Lambda^0(T) =
\#V$. To ensure $H(\curl)$ and $H(\div)$ conformity, which in turn
requires tangential components and normal components be continuous
across interfaces, respectively, our conjecture is that
$$
\min \left(\dim \mathcal{M}\Lambda^1(T)\right) = \#E,\qquad \min
\left(\dim \mathcal{M}\Lambda^2(T)\right) = \#F,
$$
which remains to be verified later by construction. According to
Euler's formula for convex polyhedra, one has
$$
\#E = \#V+\#F-2 = (\#V-1) + (\#F-1).
$$
This helps to formulate an exact sequence that we aim to build:
\begin{equation} \label{eq:minimalComplex3D}
0 \xrightarrow{\hspace*{0.2cm}} \bbR \xrightarrow[\al]{\subset}
\begin{matrix} \mathcal{M}\Lambda^0(T)
\\{\scriptstyle  dim = \#V} \end{matrix} \xrightarrow[\al]{\grad}
\begin{matrix}\mathcal{M}\Lambda^1(T) \\ {\scriptstyle dim = \#E
}\\ {\scriptstyle = (\#V-1)+(\#F-1)}
\end{matrix}
\xrightarrow[\al]{\curl} \begin{matrix} \mathcal{M}\Lambda^2(T)
\\{\scriptstyle dim=\#F}\end{matrix} \xrightarrow[\al]{\div} \bbR \xrightarrow{\hspace*{0.2cm}}
0 .
\end{equation}
Analogously, when $T$ is a 2D polygon, we aim at building an exact
sequence
\begin{equation} \label{eq:minimalComplex2D}
0 \xrightarrow{\hspace*{0.2cm}} \bbR \xrightarrow[\al]{\subset}
\begin{matrix} \mathcal{M}\Lambda^0(T)
\\{\scriptstyle  dim = \#V} \end{matrix} \xrightarrow[\al]{\curl}
\begin{matrix}\chi(\mathcal{M}\Lambda^1(T))\\ {\scriptstyle dim = \#E =\#V} \end{matrix}
\xrightarrow[\al]{\div} \bbR \xrightarrow{\hspace*{0.2cm}} 0.
\end{equation}
In the rest of this paper, we shall focus on constructing
$\chi(\mathcal{M}\Lambda^1(T))$ in 2D, as well as $\cM\Lambda^1(T)$
and $\cM\Lambda^2(T)$ in 3D, that make sequences
(\ref{eq:minimalComplex3D})-(\ref{eq:minimalComplex2D}) exact, and
more importantly, allows one to build $H(\curl)$ and $H(\div)$
conforming finite element spaces.

The rest of the paper is organized as follows. We briefly introduce
the definition and properties of the generalized barycentric
coordinates in Section \ref{sec:barycentric}. Assumptions on the
polytope $T$ and the generalized barycentric coordinates will also
be stated in this section. Then, in Section \ref{sec:2D}, we
construct $H(\div)$ conforming element
$\chi(\mathcal{M}\Lambda^1(T))$ for arbitrary convex polygons in 2D,
which satisfies (\ref{eq:minimalComplex2D}). Our formula is
different from, and easier to compute in practice than the 2D
formula given in \cite{Euler06}, although the resulting basis
functions may be identical. Moreover, when the polygon satisfy
certain shape regularity conditions, we prove the optimal mixed
finite element a priori error. Numerical results are presented too.
In Section \ref{sec:3D}, we construct $H(\curl)$ conforming element
$\cM\Lambda^1(T)$ and $H(\div)$ conforming element $\cM\Lambda^2(T)$
in 3D, which satisfy (\ref{eq:minimalComplex3D}). The current
construction only works for polyhedra whose faces are either
triangles or parallelograms. Examples show that our construction, as
one unified formula, reproduces existing minimal degree finite
elements on tetrahedra, rectangular boxes, pyramids, and triangular
prisms. We also construct finite elements on a regular octahedron,
which has never been done before. Moreover, for certain type of
polyhedra, we prove that $\cP_1^-\Lambda^k(T)\subset
\cM\Lambda^k(T)$, for $k=0,1,2$, which will ensure the approximation
property of $\cM\Lambda^k(T)$.

\section{Generalized barycentric coordinates and assumptions} \label{sec:barycentric}
Let $T$ be a convex polygon or polyhedron with $n$ vertices denoted
by $\vv_i$, for $i=1,\ldots,n$. The generalized barycentric
coordinates  are functions $\lambda_i$, for $i=1,\ldots,n$, that
satisfy:
\begin{enumerate}
\item (Non-negativity) All $\lambda_i$, for $1\le i\le n$, have
non-negative value on $T$;
\item (Linear precision) For any linear function $L(\vx)$ defined on $T$, one has
$$
L(\vx) = \sum_{i=1}^n L(\vv_i) \lambda_i(\vx),\qquad \textrm{for all
}\vx\in T.
$$
\end{enumerate}
The linear precision property is indeed equivalent to the
combination of the following two properties: for all $\vx\in T$,
\begin{equation} \label{eq:GBCLinearPrecision}
\sum_{i=1}^n \lambda_i (\vx) = 1, \qquad\quad \sum_{i=1}^n \lambda_i
(\vx) \,\vv_i = \vx.
\end{equation}

Different types of generalized barycentric coordinates have been
proposed in both 2D and 3D. Reader's may refer to \cite{Floater03,
Floater05, Floater15, Joshi07, Martin08, Sibon90, Wachspress75,
Wachspress11, Warren96} and references therein for more details.
When $T$ is a simplex, all generalized barycentric coordinates are
identical, and they are equal to the traditional barycentric
coordinates on simplices, which span the space of all linear
polynomials.

The spaces $\cM\Lambda^k(T)$ that we plan to construct in this paper
will be based on generalized barycentric coordinates. In the
construction, we do require certain properties from generalized
barycentric coordinates, which will be listed below as an
assumption. We will also explain that the following assumption is
not unreasonable, since there exist generalized barycentric
coordinates that satisfy all terms in the assumption. But here we
choose to list them as assumptions instead of limiting our interest
to specific coordinates, in order to provide a more general setting.

\smallskip
{\bf Assumption 1:} {\em There exists a set of generalized
barycentric coordinates on $T$ satisfying the following:
\begin{itemize}
\item (Lagrange property) For all $1\le i,j\le n$, one has $\lambda_i(\vv_j) = \delta_{ij}$, where
$\delta_{ij}$ is the Kronecker delta;
\item (Trace property) In 2D, each $\lambda_i$ is piecewise linear on $\partial T$. In 3D,
each $\lambda_i$ degenerates into a 2D generalized barycentric
coordinate satisfying Assumption 1 on each face of $T$.
\item (Smoothness) For all $1\le i\le n$, one has $\lambda_i\in
C^1(T)$.
\end{itemize}}

\begin{rem} Assumption 1 is not unreasonable. It has been proved in \cite{Floater06} that all 2D
generalized barycentric coordinates on convex polygons satisfy the
Lagrange property and the trace property. In 3D, the Wachspress
coordinates \cite{Warren96} and the mean value coordinates
\cite{Floater05} have been defined and studied. The Wachspress
coordinates satisfy the Lagrange property and the trace property on
all convex polytopes \cite{WarrenUniqueness}. The mean value
coordinates have been proved to satisfy the Lagrange property and
the trace property on convex polytopes whose faces are all
triangular \cite{Floater05}. Both the Wachspress and the mean value
coordinates are known to be in $C^{\infty}$ in the interior of $T$
and have unique continuous extension to $\partial T$.
\end{rem}

In 3D, we will need to impose an additional assumption on the convex
polyhedron $T$, which basically requires each face of $T$ must be
either a triangle or a parallelogram. To explain the reason for such
a restrictive assumption, we first list some special properties of
2D generalized barycentric coordinates on triangles and
parallelograms. Denote by $|\cdot|$ the length/area/volume of an
edge/polygon/polyhedron, depending on the context.

\begin{lem} \label{lem:triangleCross}
Consider a triangle $T$ with vertices $\vv_i$, $1\le i \le 3$,
ordered counter-clockwisely. Denote the barycentric coordinates by
$\lambda_i$, $1\le i\le 3$. Their gradients $\nabla \lambda_i$ are
two-dimensional constant vectors. We have
\begin{equation} \label{eq:triangleface}
det\begin{bmatrix}\nabla\lambda_i\; \nabla\lambda_j\end{bmatrix}
\equiv \frac{1}{2|T|},
\end{equation}
for $(i,j)\in \{(1,2),\, (2,3),\, (3,1)\}$.
\end{lem}
\begin{proof}
Denote by $e_i$ the edge opposite to vertex $\vv_i$, for $1\le i\le
3$. Clearly, $\nabla\lambda_i$ is a constant vector orthogonal to
$e_i$, pointing from $e_i$ towards $\vv_i$, and with length
$\frac{|e_i|}{2|T|}$. Denote by $\theta_{ij}$ the internal angle of
$T$ formed by edges $e_i$ and $e_j$. Then we have
$$
\begin{aligned}
det\begin{bmatrix}\nabla\lambda_i\; \nabla\lambda_j\end{bmatrix} &=
|\nabla\lambda_i|\,|\nabla\lambda_j|\, \sin(\pi-\theta_{ij}) \\ &=
\frac{|e_i|}{2|T|}\, \frac{|e_j|}{2|T|} \, \sin \theta_{ij} =
\frac{2|T|}{(2|T|)^2} = \frac{1}{2|T|}.
\end{aligned}
$$
This completes the proof of the lemma.
\end{proof}

\begin{lem} \label{lem:parallelogramCross}
Consider a parallelogram $T$ with vertices $\vv_i$, $1\le i \le 4$,
ordered counter-clockwisely. Denote the Wachspress coordinates on
$T$ by $\lambda_i$, $1\le i\le 4$. Their gradients $\nabla
\lambda_i$ are two-dimensional vectors. We have
\begin{equation} \label{eq:parallelogramface}
\begin{aligned}
det\begin{bmatrix}\nabla\lambda_1\; \nabla\lambda_2\end{bmatrix} + det\begin{bmatrix}\nabla\lambda_3\; \nabla\lambda_4\end{bmatrix} &\equiv \frac{1}{|T|}, \\
det\begin{bmatrix}\nabla\lambda_2\; \nabla\lambda_3\end{bmatrix} +
det\begin{bmatrix}\nabla\lambda_4\; \nabla\lambda_1\end{bmatrix}
&\equiv \frac{1}{|T|}.
\end{aligned}
\end{equation}
\end{lem}
\begin{proof}
Without loss of generality, denote the vertices of $T$, in
counter-clockwise order, by $\vv_1: (0,0)$, $\vv_2: (h_1,0)$,
$\vv_3: (h_1+kh_2, h_2)$, $\vv_4: (kh_2, h_2)$, where $h_1$, $h_2$
and $k$ are positive constants. Then, one can easily compute the
Wachspress coordinates and their gradients:
$$
\begin{aligned}
\lambda_1 &= \frac{(h_1 - x+ky)(h_2-y)}{h_1h_2}, \qquad &\nabla
 \lambda_1 &= [\frac{-(h_2-y)}{h_1h_2},\,
 \frac{x-2ky-h_1+kh_2}{h_1h_2} \,]^t, \\
 \lambda_2 &= \frac{(x-ky)(h_2-y)}{h_1h_2}, \qquad &\nabla
   \lambda_2 &= [\frac{h_2-y}{h_1h_2},\,
 \frac{-x+2ky-kh_2}{h_1h_2} \,]^t, \\
\lambda_3 &= \frac{(x-ky)y}{h_1h_2}, \qquad &\nabla
  \lambda_3 &= [\frac{y}{h_1h_2},\,
 \frac{x-2ky}{h_1h_2} \,]^t, \\
\lambda_4 &= \frac{(h_1 - x+ky)y}{h_1h_2}, \qquad &\nabla
  \lambda_4 &= [\frac{-y}{h_1h_2},\,
 \frac{-x+2ky+h_1}{h_1h_2} \,]^t.
\end{aligned}
$$
The lemma hence follows from direct calculation.
\end{proof}

Now we state the additional assumption on $T$:

\smallskip
{\bf Assumption 2:} {\em In 3D, assume each face of polyhedron $T$
be either a triangle or a parallelogram. Moreover, assume the trace
of the generalized barycentric coordinates chosen in our
construction satisfy equations
(\ref{eq:triangleface})-(\ref{eq:parallelogramface}) on the faces of
$T$.}
\smallskip

\begin{rem}
Equations (\ref{eq:triangleface})-(\ref{eq:parallelogramface}) will
later ensure that each function in the constructed $H(\div)$ finite
element space has constant normal components on faces. This is why
we need Assumption 2. Similar but much more complicated equations,
with non-constant right-hand sides, can be obtained for general
polygons. Whether they can be used to build vector-valued finite
elements on polyhedra not satisfying Assumption 2 is a topic for
future research.
\end{rem}

\begin{rem}
According to lemmas
\ref{lem:triangleCross}-\ref{lem:parallelogramCross}, for convex
polyhedra with only triangular faces, both the Wachspress and the
mean value coordinates can be used in the construction; while for
convex polyhedra with both triangular faces and parallelogramal
faces, only the Wachspress coordinates can be used.
\end{rem}

Throughout the rest of this paper, we always assume the polytope, as
well as the generalized barycentric coordinates defined on it,
satisfy Assumptions 1-2. It is known that all polygons and many
polyhedra, including the most frequently used tetrahedra,
parallelepipeds, triangular prisms, and pyramids, have generalized
barycentric coordinates defined on them that satisfy these
assumptions.

\section{Construction in 2D} \label{sec:2D}
Let $T$ be a convex polygon. Denote by $\vv_i$, $1\le i\le n$, the
vertices of $T$ ordered counterclockwisely, and by $e_i$ the edge
connecting vertices $\vv_i$ and $\vv_{i+1}$, where we conveniently
denote $\vv_{j}=\vv_{j\, (mod\, n)}$ when the subscript $j$ is not
in the range of $\{1,\ldots, n\}$. Similar tricks of indexing will
be used frequently without special mentioning. Denote by $\vn_i$ and
$\vt_i$ the unit outward normal and the unit tangent vector in the
counterclockwise orientation on $e_i$. Choose an arbitrary point
$\vx_*$ inside polygon $T$, and denote by $T_i$ the triangle with
base $e_i$ and apex $\vx_*$. Denote by $d_i$ the distance from
$\vx_*$ to $e_i$. Let $|e_i|$, $|T_i|$ and $|T|$ be the length of
$e_i$, the area of $T_i$ and $T$, respectively. It is clear that
$|T_i|=\frac{1}{2}|e_i|d_i$ and $|T|=\sum_{i=1}^n |T_i|$. We use the
standard notation $L^p(T)$, $W^{s,p}(T)$, $H^s(T)$ and $H(\div,T)$,
with $s\in \bbR$ and $1\le p\le \infty$ for different type of
Sobolev spaces, equipped with corresponding innerproducts and norms.
For simplicity, denote by $\|\cdot\|_T$ and $\|\cdot\|_{e_i}$ the
$L^2$ norm on $T$ and $e_i$ respectively, while by $\|\cdot\|_{1,T}$
the $H^1$ norm on $T$. Finally, denote by $h_T$ the diameter of $T$.

\subsection{Discrete space and basis function}
Recall that $\mathcal{M}\Lambda^0(T) = span\{\lambda_i,\,
i=1,\ldots,n\}$. By (\ref{eq:GBCLinearPrecision}), one has
$\bbR\subset \mathcal{M}\Lambda^0(T)$ and thus the sequence
(\ref{eq:minimalComplex2D}) is obviously exact at the
$\cM\Lambda^0(T)$ node. In order to ensure the exactness at the
$\chi(\mathcal{M}\Lambda^1(T))$ node, we would like to define
$\chi(\mathcal{M}\Lambda^1(T))$ with an orthogonal decomposition,
i.e., the discrete Helmholtz decomposition:
$$
\chi(\mathcal{M}\Lambda^1(T)) = \curl \mathcal{M}\Lambda^0(T) \oplus
(\div^{\dagger}) \bbR,
$$
where $\div^{\dagger}$ stands for a pseudo-inverse of $\div$ under
proper choice of spaces such that $(\div^{\dagger}) \bbR$ contains
functions orthogonal to $\curl \mathcal{M}\Lambda^0(T)$ and with
divergence in $\bbR$. In practice, it is much easier if one relaxes
the orthogonality a little bit through replacing $\oplus$ by $+$,
and thus we consider the following construction:
\begin{equation} \label{eq:HdivSpace2D}
\begin{aligned}
\chi(\mathcal{M}\Lambda^1(T)) &= \curl \mathcal{M}\Lambda^0(T) +
span\{
\vx - \vx_* \} \\
&= span\{\curl \lambda_i, \, i=1,\ldots, n\} + span\{ \vx - \vx_*
\}.
\end{aligned}
\end{equation}
Later we shall show that the above definition is independent of the
choice of $\vx_*$.

By construction, it is clear that $\div \,
\chi(\mathcal{M}\Lambda^1(T)) = \bbR$ and $\curl
\mathcal{M}\Lambda^0(T)\cap span\{ \vx - \vx_* \} = \{\mathbf{0}\}$.
Therefore, the sequence (\ref{eq:minimalComplex2D}) is also exact at
the $\chi(\mathcal{M}\Lambda^1(T))$ node. Now we know that the
entire sequence (\ref{eq:minimalComplex2D}) is exact. By counting
dimensions and since obviously $dim \mathcal{M}\Lambda^0(T) = n$,
one must have $dim\, \chi(\mathcal{M}\Lambda^1(T)) = n$. Next, we
explicitly construct a set of basis for
$\chi(\mathcal{M}\Lambda^1(T))$.

For $1\le i,l\le n$, define
$$b_{i,l} = \delta_{il}|e_l| - |e_i|\frac{|T_l|}{|T|}.$$
The above notation can be extended to indices not in $\{1,\ldots,
n\}$ using modular arithmetic.

\begin{lem}
For each $1\le i\le n$, define $\vq_i\in
\chi(\mathcal{M}\Lambda^1(T))$ by
\begin{equation} \label{eq:2Dbasis}
\vq_i = c_{i,0} (\vx-\vx_*) + \sum_{k=1}^n c_{i,k} \curl \lambda_k,
\end{equation}
where $ c_{i,0} = \frac{|e_i|}{2|T|}$ and $c_{i,k} = -\frac{1}{n}
\sum_{l=1}^{n-1} l \, b_{i,k+l}$. Then, one has
$\vq_i\cdot\vn_j|_{e_j} \equiv \delta_{ij}$ for all $1\le j\le n$,
and the set $\{\vq_i, \, 1\le i\le n\}$ form a basis for
$\chi(\mathcal{M}\Lambda^1(T))$.
\end{lem}
\begin{proof}
Notice that for all $1\le k\le n$, one has
$$
\sum_{l=1}^n b_{i,k+l} = \sum_{l=1}^n b_{i,l} = |e_i| - |e_i|
\sum_{l=1}^n \frac{|T_l|}{|T|} = 0,
$$
which implies that
$$
\begin{aligned}
c_{i,k}-c_{i,k+1} &= -\frac{1}{n} \left(\sum_{l=1}^{n-1} l\,
b_{i,k+l} - \sum_{l=1}^{n-1} l\, b_{i,k+l+1}\right)
= -\frac{1}{n} \left(\sum_{l=1}^n b_{i, k+l} - n b_{i,k +n } \right) \\
&= -\frac{1}{n} \left(0 - n b_{i,k } \right) = b_{i,k}.
\end{aligned}
$$
Therefore, by the definition of generalized barycentric coordinates
and Assumption 1, we have
$$
\begin{aligned}
\vq_i\cdot \vn_j|_{e_j} &= c_{i,0} (\vx-\vx_*)\cdot \vn_j|_{e_j} +
\sum_{k=1}^n c_{i,k} \curl \lambda_k \cdot \vn_j|_{e_j} = c_{i,0}\,
d_j
   - \sum_{k=1}^n c_{i,k} \left.\frac{\partial \lambda_k}{\partial
   \vt_j}\right|_{e_j} \\
&\equiv c_{i,0} \frac{2|T_j|}{|e_j|} - \left(-\frac{c_{i,j}}{|e_j|}
+ \frac{c_{i,j+1}}{|e_j|} \right) = \frac{1}{|e_{j}|}
\left(\frac{|e_i|}{|T|}|T_j| + b_{i,j} \right)
\\
&= \delta_{ij}.
\end{aligned}
$$
The set $\{\vq_i,\,1\le i\le n\}$ is linearly independent, because
$\sum_{i=1}^n a_i \vq_i = \mathbf{0}$ implies that $0 =
(\sum_{i=1}^n a_i \vq_i )\cdot \vn_j|_{e_j} = a_j$ for all $1\le
j\le n$.  Since $\dim\, \chi(\mathcal{M}\Lambda^1(T))=n$, the set
$\{\vq_i,\,1\le i\le n\}$ must form a basis for
$\chi(\mathcal{M}\Lambda^1(T))$. This completes the proof of the
lemma.
\end{proof}

\begin{rem}
From the basis it is clear that for any function $\vq\in
\chi\left(\mathcal{M}\Lambda^1(T)\right)$, the normal component
$\vq\cdot\vn$ is piecewise constant on $\partial T$. Moreover, the
normal components on edges form a unisolvant set of degrees of
freedom for $\chi\left(\mathcal{M}\Lambda^0(T)\right)$. Such a
choice of degrees of freedom guarantees that one can build $H(\div)$
conforming finite element spaces on general polygonal meshes using
$\chi\left(\mathcal{M}\Lambda^1(T)\right)$.
\end{rem}

\begin{rem}
When $T$ is a triangle, all currently known generalized barycentric
coordinates degenerate to the unique triangular barycentric
coordinates $\{\lambda_1,\, \lambda_2,\, \lambda_3\}$. In this case,
the space $\mathcal{M}\Lambda^0(T)$ is identical to $span\{1,x,y\}$,
and consequently the space
$\chi\left(\mathcal{M}\Lambda^1(T)\right)$ is identical to
$\mathcal{P}_1^{-}\Lambda^1(T)$, the lowest-order Raviart-Thomas
finite element on triangles. When $T$ is a rectangle and
$\lambda_i$'s are chosen to be the Wachspress coordinates, the space
$\mathcal{M}\Lambda^0(T)$ is identical to $span\{1,x,y,xy\}$, and
consequently $\chi\left(\mathcal{M}\Lambda^1(T)\right)$ is identical
to $\mathcal{H}_1\Lambda^1(T)$, the lowest-order Raviart-Thomas
finite element on rectangles. In this sense, the space
$\chi\left(\mathcal{M}\Lambda^1(T)\right)$ can be viewed as the
extension of the lowest-order Raviart-Thomas finite element to
general polygons.
\end{rem}

A more important relation between $\mathcal{P}_1^{-}\Lambda^1(T)$
and $\chi\left(\mathcal{M}\Lambda^1(T)\right)$ is given in the
following lemma:

\begin{lem} \label{lem:RT0included2D}
The space $\chi\left(\mathcal{M}\Lambda^1(T)\right)$ reproduces all
functions in $\mathcal{P}_1^{-}\Lambda^1(T)$, i.e.,
$$
\mathcal{P}_1^{-}\Lambda^1(T) \subseteq
\chi\left(\mathcal{M}\Lambda^1(T)\right).
$$
\end{lem}
\begin{proof}
This follows immediately from the definition of
$\chi\left(\mathcal{M}\Lambda^1(T)\right)$ and the fact that
$\cP_1^-\Lambda^0(T)\subseteq \mathcal{M}\Lambda^0(T)$, which comes
from Equation (\ref{eq:GBCLinearPrecision}).
\end{proof}

\begin{rem}
Lemma \ref{lem:RT0included2D} indicates that the space
$\chi\left(\mathcal{M}\Lambda^1(T)\right)$ is independent of the
choice of $\vx_*$.
\end{rem}

Finally, we briefly show that $\chi(\cM\Lambda^1(T))$ constructed in
this section is a subspace of $\chi(\mathcal{W}\Lambda^1(T))$. By
Equation (\ref{eq:GBCLinearPrecision}), it is not hard to see that
\begin{equation} \label{eq:Wijsummation}
\sum_{j=1}^n W_{ij} = -\nabla \lambda_i,\qquad\textrm{for all }1\le
i\le n,
\end{equation}
which implies that $\curl \cM\Lambda^0(T) = \chi (\nabla
\cM\Lambda^0(T) ) \subseteq \chi(\mathcal{W}\Lambda^1(T))$. Recall
the inclusion relation of finite elements in
(\ref{eq:WLambdaTincludeP-LambdaT}), one has
$\cP_1^-\Lambda^1(T)\subseteq \chi(\mathcal{W}\Lambda^1(T))$.
Combining the above with the definition of $\chi(\cM\Lambda^1(T)$
gives $\chi(\cM\Lambda^1(T)\subseteq \chi(\mathcal{W}\Lambda^1(T))$.

\subsection{Interpolation operator and its properties} To make sure that the mixed finite element
theory works on the finite element
$\chi\left(\mathcal{M}\Lambda^1(T)\right)$, we define an
interpolation operator into
$\chi\left(\mathcal{M}\Lambda^1(T)\right)$ which satisfies certain
stability and approximation properties. For convenience, we
introduce the notation $\lesssim$, $\gtrsim$ and $\approx$ for `less
than or equal to', `greater than or equal to', `both less than or
equal to and greater than or equal to' up to a constant independent
of the shape of all polygons in a given mesh.

Clearly, to establish any kind of stability and approximation
properties, the polygonal mesh needs to satisfy certain shape
regularity conditions. We assume for all polygons in the mesh,
\begin{itemize}
\item The area of the polygon is related to its diameter as follows:
$$
|T|\approx h_T^2;
$$
\item The gradient of $\lambda_i$, for $1\le i\le n$, on $T$ satisfies
\begin{equation} \label{eq:gradlambdaupperbound}
|\nabla \lambda_i |\lesssim h_T^{-1}, \qquad\textrm{at all }\vx\in
T,
\end{equation}
where $|\cdot|$ stands for the Euclidean length. It has been proved
in \cite{Floater14} that (\ref{eq:gradlambdaupperbound}) holds for
Wachspress coordinates as long as $h^*$, the minimum distance from
any vertex of $T$ to a non-incidental edge, satisfies $h^*\approx
h_T$.
\item The following trace inequality and approximation property of $L^2$
projection hold on $T$:
\begin{equation} \label{eq:traceapproximation}
\begin{aligned}
\|\phi\|_{L^2(\partial T)}^2 &\lesssim h_T^{-1} \|\phi\|_{T}^2 + h_T
\|\nabla\phi \|_{T}^2 ,\qquad &&\textrm{for }\phi\in
H^1(T), \\
\|\phi - P_T \phi\|_{T} &\lesssim h_T \|\phi\|_{1,T}, \qquad
&&\textrm{for }\phi\in H^1(T),
\end{aligned}
\end{equation}
where $P_T$ denotes the $L^2$ orthogonal projection onto $\bbR$. It
is known that when $T$ satisfy certain shape regularity conditions,
(\ref{eq:traceapproximation}) holds on $T$. Readers may refer to
\cite{Brezzi05, Herbin96, MuWangWang15, Wang14} for further
discussion.
\end{itemize}
\smallskip


Now let use define the interpolation operator. For any $\vq\in
H(\div,\, T)\cap (L^p(T))^2$ with $p>2$, define $\Pi_T\vq\in
\chi\left(\mathcal{M}\Lambda^1(T)\right)$ by
$$
(\Pi_T\vq)\cdot \vn_j|_{e_j} = \frac{1}{|e_j|} \int_{e_j}
\vq\cdot\vn_j \, ds,\qquad \textrm{for }1\le j\le n.
$$
The requirement $p>2$ is to guarantee that $\int_{e_j} \vq\cdot\vn_j
\, ds$ be well-defined. One may circumvent this requirement by using
Cl\'{e}ment type interpolations \cite{Clement75}. According to the
definition, it is clear that
$$
\Pi_T \vq = \sum_{i=1}^n a_i \vq_i, \qquad\textrm{where }a_i =
\frac{1}{|e_i|} \int_{e_i} \vq\cdot\vn_i \, ds.
$$
Moreover, by the unisolvancy of the degrees of freedom and Lemma
\ref{lem:RT0included2D}, we know that $\Pi_T$ preserves all
functions in $\mathcal{P}_1^{-}\Lambda^1(T)$, i.e.,
$$
\Pi_T \begin{pmatrix}cx+a\\ cx+b\end{pmatrix} = \begin{pmatrix}cx+a\\
cx+b\end{pmatrix},\qquad \textrm{for all }a,b,c\in \bbR.
$$

Denote by $I_T$ the nodal value interpolation into
$\mathcal{M}\Lambda^0(T)$. Properties of nodal value interpolation
for generalized barycentric coordinates have be discussed in
\cite{Floater14, Gillette12}. Then we have:
\begin{lem} \label{lem:commutativeDiagram2D}
Let $p>2$.  For any $\vq\in H(\div,\, T)\cap (L^p(T))^2$ , one has
$\div \Pi_T \vq = P_T \div \vq$. For any $\phi\in W^{1,p}(T)$, one
has $\Pi_T\curl \phi = \curl I_T \phi$. In other words, the
following diagram is commutative:
$$
\begin{matrix}
W^{1,p}(T) & \xrightarrow[\al]{\curl} & H(\div,\, T)\cap (L^p(T))^2
& \xrightarrow[\al]{\div} & L^2(T)
\\
I_T {\bigg\downarrow} & &  \Pi_T {\bigg\downarrow} & & P_T {\bigg \downarrow}  \\
\mathcal{M}\Lambda^0(T) & \xrightarrow[\al]{\curl} &
\chi\left(\mathcal{M}\Lambda^1(T)\right) & \xrightarrow[\al]{\div} &
\bbR
\end{matrix}
$$
\end{lem}
\begin{proof}
Given $\vq\in H(\div,\, T)\cap (L^p(T))^2$, let $a_i =
\frac{1}{|e_i|} \int_{e_i} \vq\cdot\vn_i \, ds$ for $1\le i\le n$.
Then by the definition of basis function $\vq_i$, one has
$$
\begin{aligned}
\div\Pi_T \vq &= \div \sum_{i=1}^n a_i \vq_i
 =  \sum_{i=1}^n a_i \frac{|e_i|}{|T|}
 =  \sum_{i=1}^n \frac{1}{|T|} \int_{e_i} \vq\cdot\vn_i \, ds \\
 &= \frac{1}{|T|} \int_T \div \vq\, dx
 = P_T \div \vq.
\end{aligned}
$$
Given $\phi\in W^{1,p}(T)$. Then $I_T\phi = \sum_{i=1}^n \phi(\vv_i)
\lambda_i$. Note that for all $1\le j\le n$, one has
$$
\begin{aligned}
(\Pi_T \curl\phi)\cdot\vn_j|_{e_j} &= \frac{1}{|e_j|} \int_{e_j}
\curl\phi\cdot\vn_j  \, ds \\
&= -\frac{1}{|e_j|}\int_{e_j} \frac{\partial\phi}{\partial\vt_j}  \,
ds =\frac{\phi(\vv_j) - \phi(\vv_{j+1})}{|e_j|},
\end{aligned}
$$
and
$$
\begin{aligned}
(\curl I_T \phi) \cdot \vn_j|_{e_j} &= \left( \curl \sum_{i=1}^n
\phi(\vv_i) \lambda_i \right) \cdot \vn_j|_{e_j} = -\sum_{i=1}^n
\phi(\vv_i) \left.
\frac{\partial\lambda_i}{\partial\vt_j}\right|_{e_j} \\ &=
-\phi(\vv_j) \frac{\partial\lambda_j}{\partial\vt_j} -
\phi(\vv_{j+1}) \frac{\partial\lambda_{j+1}}{\partial\vt_j} =
\frac{\phi(\vv_j) - \phi(\vv_{j+1})}{|e_j|}.
\end{aligned}
$$
By the unisolvancy of the degrees of freedom, we have $\Pi_T\curl
\phi = \curl I_T \phi$. This completes the proof of the lemma.
\end{proof}

To prove the stability and approximation properties of $\Pi_T$, we
first derive the following estimate of $\vq_i$:
\begin{lem} \label{lem:qibound}
For $1\le i\le n$, one has
$$
\|\vq_i\|_T \lesssim C(n) |e_i|,
$$
where $C(n)$ is a general positive constant depending only on $n$.
\end{lem}
\begin{proof}
Note that
$$
\begin{aligned}
\|\vq_i\|_T^2 &= \| c_{i,0} (\vx-\vx_*) + \sum_{k=1}^n
c_{i,k}\curl\lambda_k\|_T^2 \\
&\le (n+1) \left( c_{i,0}^2 \|\vx-\vx_*\|_T^2 +  \sum_{k=1}^n
c_{i,k}^2
\|\nabla \lambda_k\|_T^2 \right)\\
&\triangleq (n+1) (J_0 + \sum_{k=1}^n J_k).
\end{aligned}
$$
For $J_0$, we have
$$
J_0 = \frac{|e_i|^2}{4|T|^2} \|\vx-\vx_*\|_T^2
\le \frac{|e_i|^2}{4|T|^2} |T| h_T^2 \\
\lesssim |e_i|^2.
$$
Here in the last step we used the assumption $|T|\approx h_T^2$.
Next, by (\ref{eq:gradlambdaupperbound}), we have the following
estimate for $J_k$:
$$
\begin{aligned}
J_k &= c_{i,k}^2 \|\nabla \lambda_k\|_T^2 \lesssim c_{i,k}^2
\frac{|T|}{h_T^2} \lesssim c_{i,k}^2
= \left( -\frac{1}{n} \sum_{l=1}^{n-1} l \, b_{i,k+l} \right)^2 \\
&\le \left( \frac{n-1}{2} \max_{1\le l\le n} |b_{i,l}|\right)^2
\lesssim n^2 |e_i|^2.
\end{aligned}
$$
Combining the above, we have proved the lemma.
\end{proof}

Denote by $Q_T$ the $(L^2(T))^2$ projection onto $\bbR^2$. Clearly
we have $\Pi_T Q_T \vq = Q_T \vq$. Next we prove the following
technical lemma:
\begin{lem} \label{lem:tvqbound}
For $\vq\in (H^1(T))^2$, one has
$$
\|\Pi_T (\vq-Q_T\vq)\|_T \lesssim C(n) h_T \|\vq\|_{1,T},
$$
where $C(n)$ is a general positive constant depending only on $n$.
\end{lem}
\begin{proof}
For convenience, denote $\tvq = \vq-Q_T\vq$. Then by the Schwarz
inequality and Lemma \ref{lem:qibound},
$$
\begin{aligned}
\|\Pi_T \tvq\|_T^2 &= \| \sum_{i=1}^n
\left(\frac{1}{|e_i|}\int_{e_i} \tvq\cdot\vn_i\, ds \right) \vq_i
\|_T^2 \le n \sum_{i=1}^n \left(\frac{1}{|e_i|}\int_{e_i}
\tvq\cdot\vn_i\, ds \right)^2 \|\vq_i\|_T^2 \\
&\le  n\sum_{i=1}^n \frac{\|\tvq\|_{e_i}^2 \|\vq_i\|_T^2}{|e_i|}
\lesssim C(n) \sum_{i=1}^n \left( |e_i|\|\tvq\|_{e_i}^2 \right).
\end{aligned}
$$
Then, by (\ref{eq:traceapproximation}), one has
$$
\|\tvq\|_{e_i}^2 \lesssim h_T^{-1} \|\tvq\|_T^2 + h_T
\|\nabla\tvq\|_T^2 \lesssim h_T \|\vq\|_{1,T}^2.
$$
Combining the above gives
$$
\|\Pi_T \tvq\|_T^2 \lesssim C(n) \left( \sum_{i=1}^n |e_i|\right)
h_T \|\vq\|_{1,T}^2 \lesssim C(n) h_T^2 \|\vq\|_{1,T}^2.
$$
This completes the proof of the lemma.
\end{proof}

Next we prove the following stability property of $\Pi_T$:
\begin{lem}
For $\vq\in (H^1(T))^2$, one has
$$
\|\Pi_T\vq\|_{H(\div,\, T)} \lesssim C(n) \|\vq\|_{1,T}.
$$
\end{lem}
\begin{proof}
By Lemma \ref{lem:commutativeDiagram2D}, it is clear that we only
need to prove $\|\Pi_T\vq\|_T \lesssim C(n)\|\vq\|_{1,T}$. Using the
triangle inequality, Lemma \ref{lem:tvqbound}, and the stability of
the $L^2$ projection $Q_T$, one has
$$
\begin{aligned}
\|\Pi_T\vq\|_T &\le \|\Pi_T (\vq - Q_T\vq) \|_T + \|\Pi_T Q_T \vq\|_T \\
&\lesssim C(n)h_T \|\vq\|_{1,T} + \|Q_T \vq\|_T \\
&\lesssim C(n) \|\vq\|_{1,T}.
\end{aligned}
$$
In the above we have used the fact that $\Pi_T Q_T \vq=Q_T \vq$.
This completes the proof of the lemma.
\end{proof}

Finnally, we prove the approximation property of $\Pi_T$:
\begin{lem}
For all $\vq\in (H^1(T))^2$, one has
$$
\|\vq-\Pi_T \vq\|_T \lesssim C(n) h_T\|\vq\|_{1,T}.
$$
Moreover, if $\div\vq\in H^1(T)$, then one has
$$
\|\div(\vq-\Pi_T\vq)\|_T \lesssim h_T \|\div\vq\|_{1,T}.
$$
\end{lem}
\begin{proof}
By the triangle inequality,  the fact that $\Pi_T Q_T \vq=Q_T \vq$,
Lemma \ref{lem:tvqbound}, and the approximation property of $Q_T$
(similar to (\ref{eq:traceapproximation})), one has
$$
\begin{aligned}
\|\vq - \Pi_T \vq\|_T &\lesssim \|\vq - Q_T\vq\|_T + \|\Pi_T
(\vq-Q_T\vq)\|_T \\
&\lesssim C(n) h_T \|\vq\|_{1,T}.
\end{aligned}
$$
The second part of the lemma follows from Lemma
\ref{lem:commutativeDiagram2D} and Inequality
(\ref{eq:traceapproximation}).
\end{proof}

\begin{rem}
Because of the above properties of $\Pi_T$, the finite element
$\chi\left(\mathcal{M}\Lambda^1(T)\right)$ fits the theoretical
framework of mixed finite element methods in the book by Brezzi and
Fortin \cite{BrezziFortin91}, as long as the polygonal mesh
satisfies all shape regularity assumptions and the number of
vertices in each polygon is bounded above. In this case, the mixed
finite element achieves optimal approximation error in both
$\|\cdot\|_{L^2(\Omega)}$ and $\|\div(\cdot)\|_{L^2(\Omega)}$, where
$\Omega$ denotes the entire computational domain.
\end{rem}

\subsection{Numerical results}
In this section, we first draw a set of basis $\{\vq_i\}$ for
$H(\div)$ element on a random pentagon in Figure
\ref{fig:pentagonbasis}, in order to give the reader a direct
picture of these basis functions. The basis is generated using the
formula (\ref{eq:2Dbasis}), with $\lambda_i$ set as the Wachspress
coordinates.

\begin{figure}[h]
\begin{center}
\includegraphics[width=3cm]{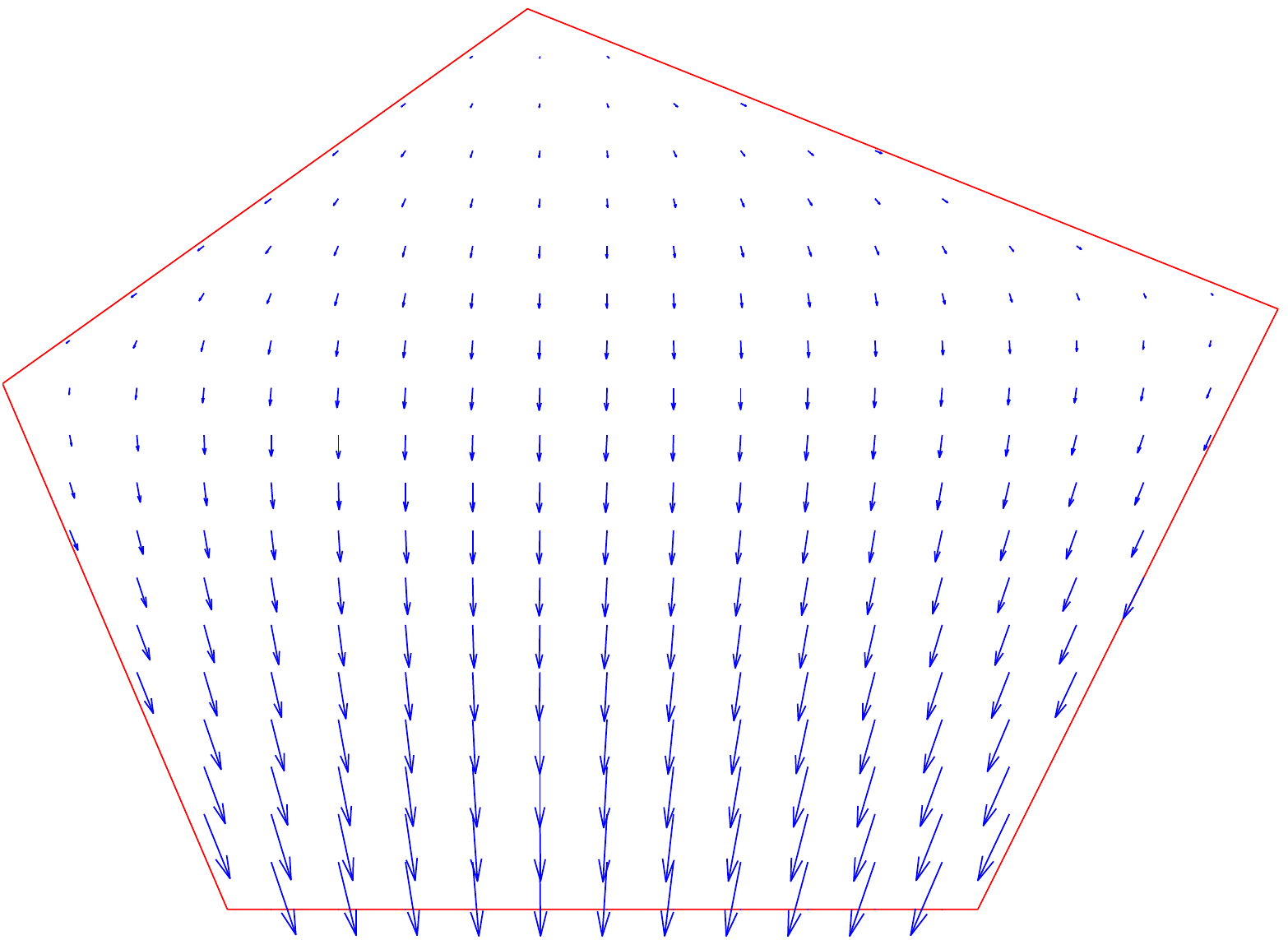}
\includegraphics[width=3cm]{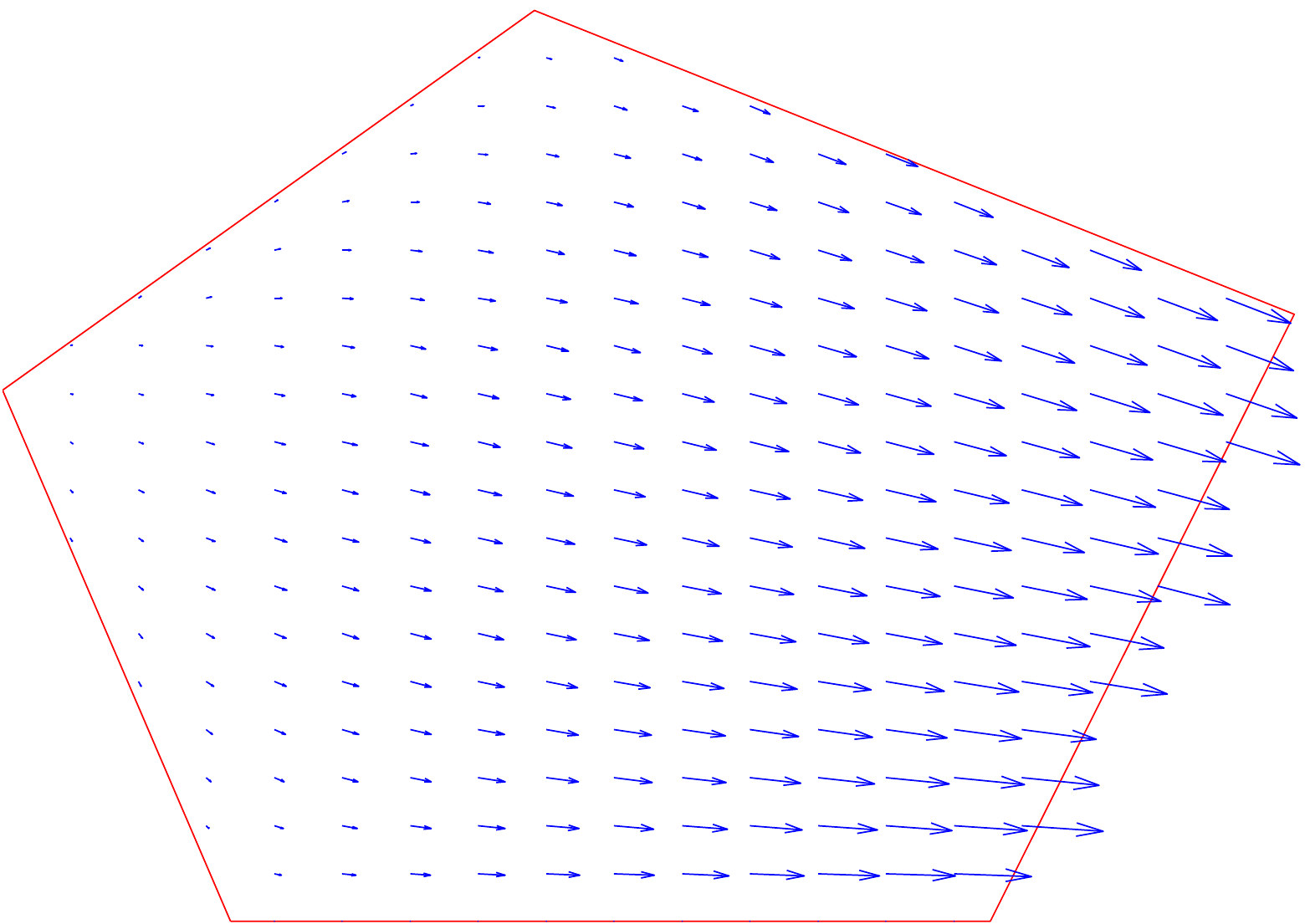}
\includegraphics[width=3cm]{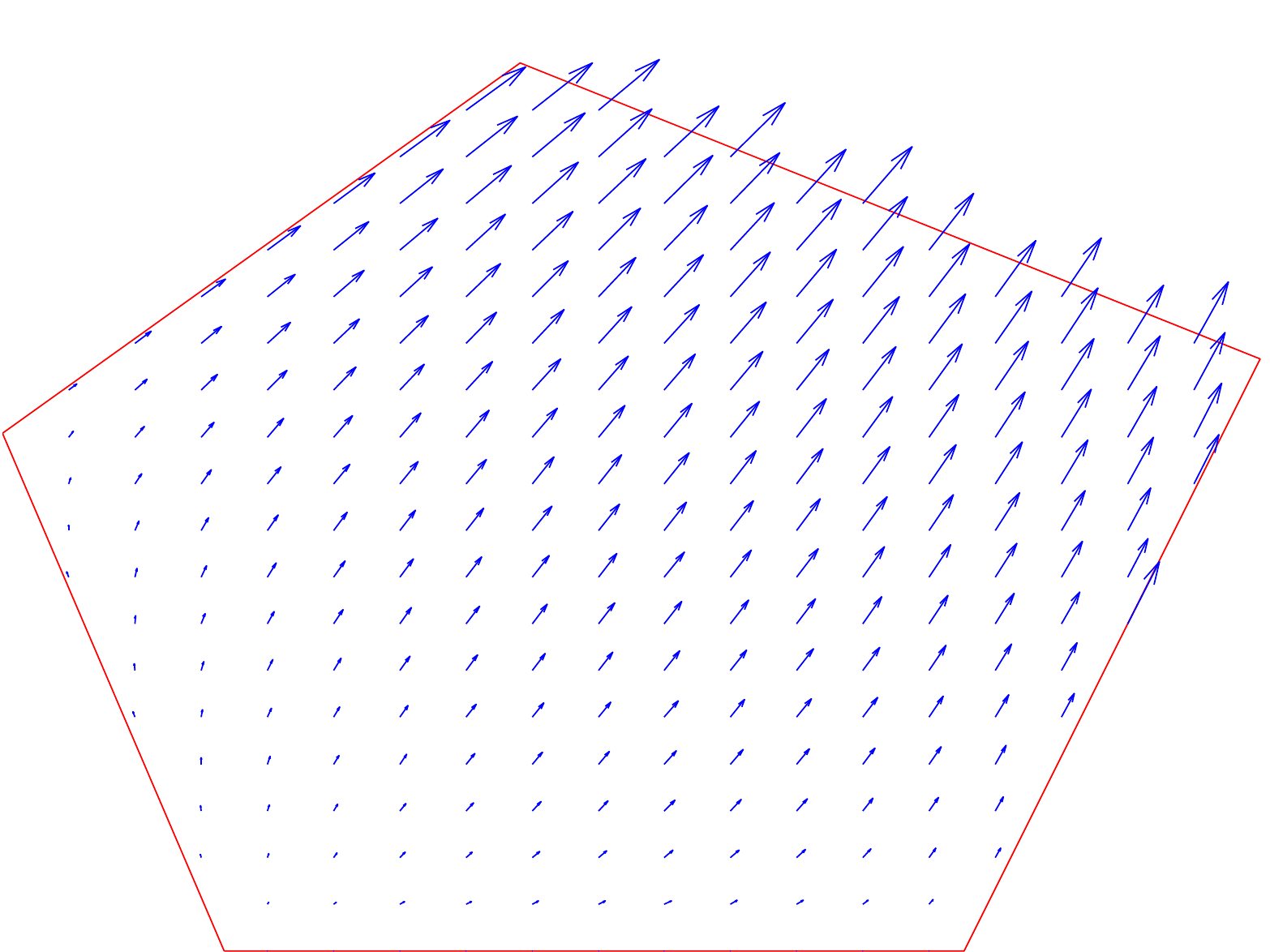}\\
\includegraphics[width=3cm]{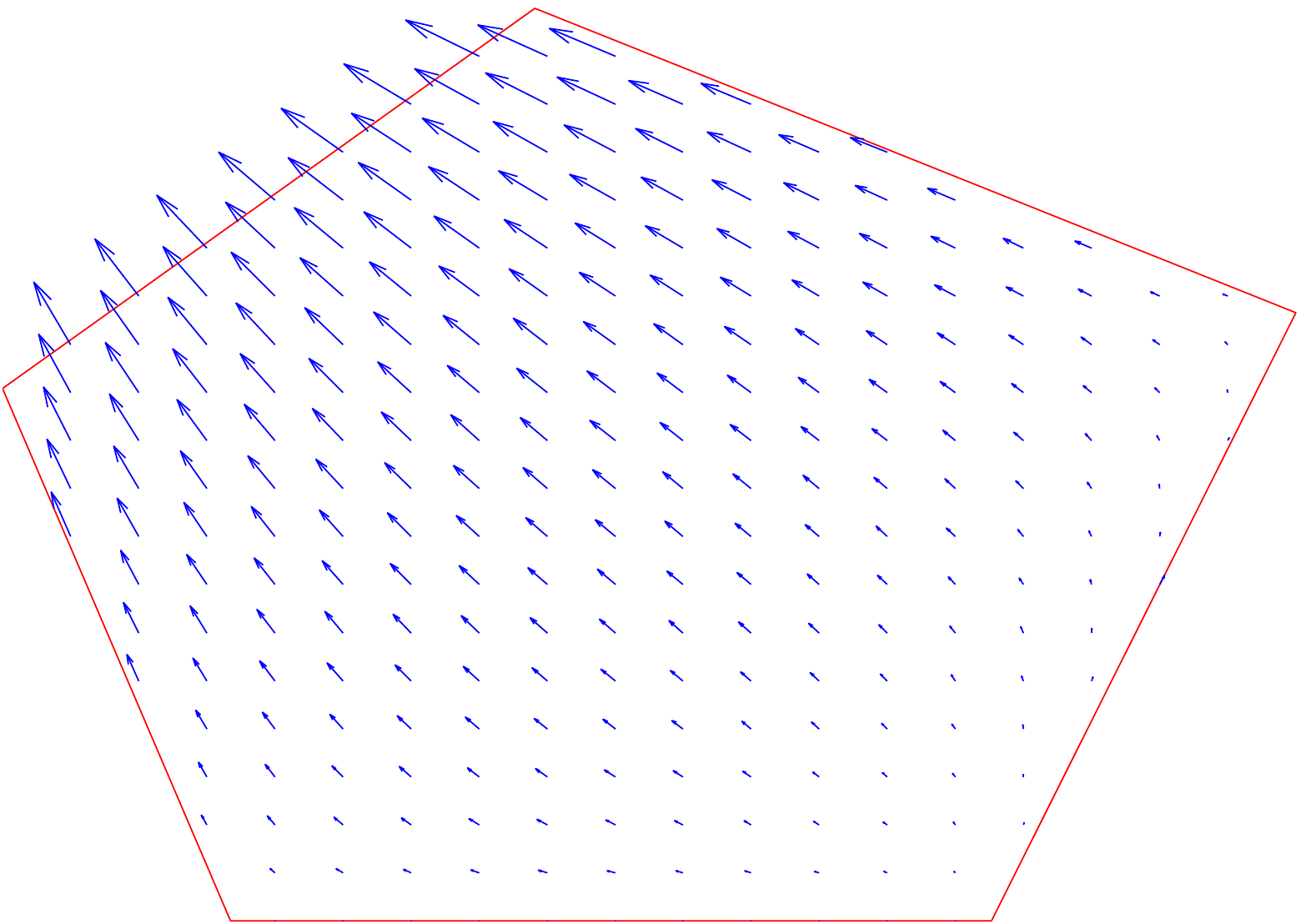}
\includegraphics[width=3cm]{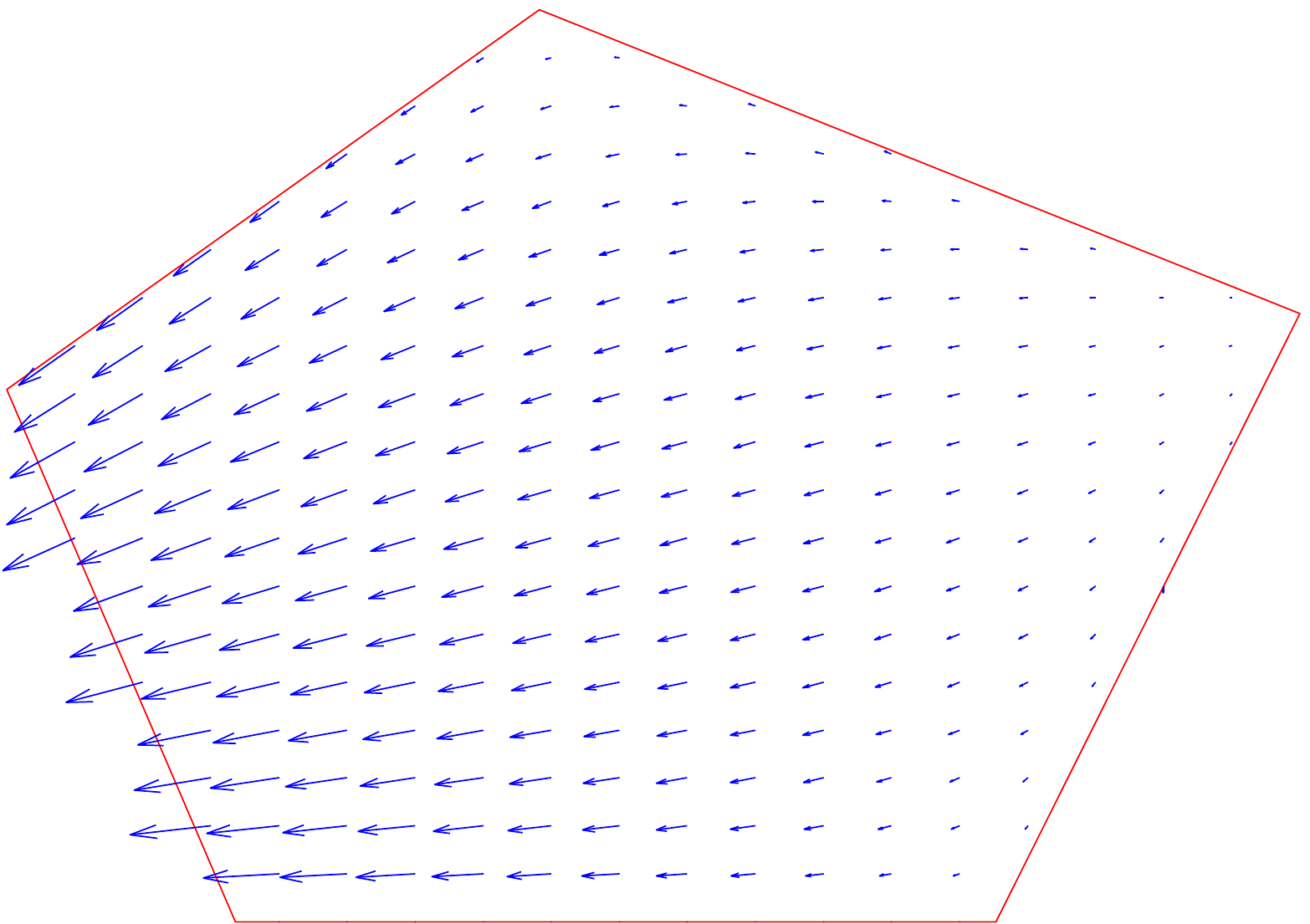}
\caption{Basis $\{\vq_i\}$ for $H(\div)$ element on a random
pentagon.} \label{fig:pentagonbasis}
\end{center}
\end{figure}

\begin{figure}[h]
\begin{center}
\includegraphics[width=3.5cm]{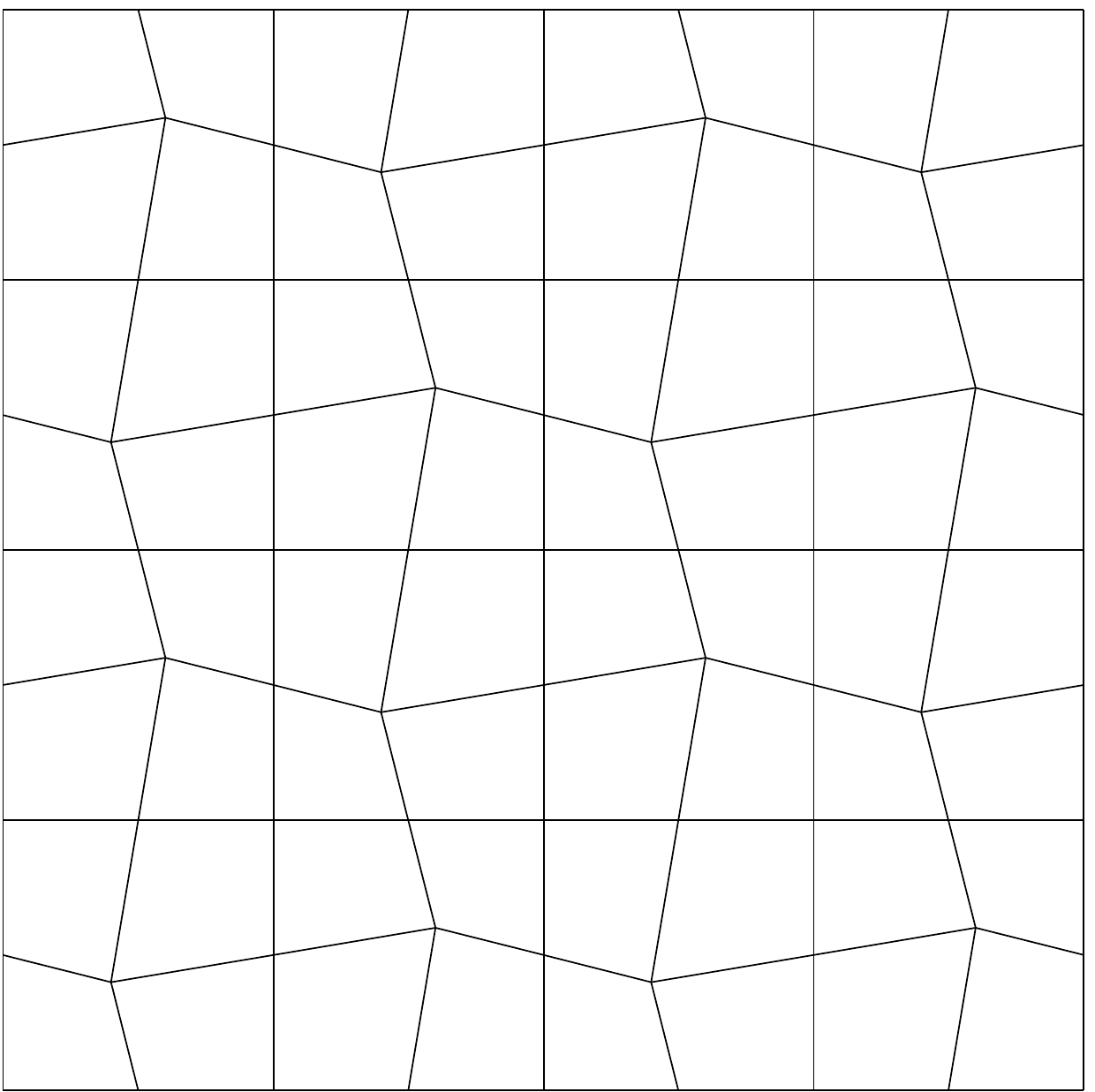}\;
\includegraphics[width=3.5cm]{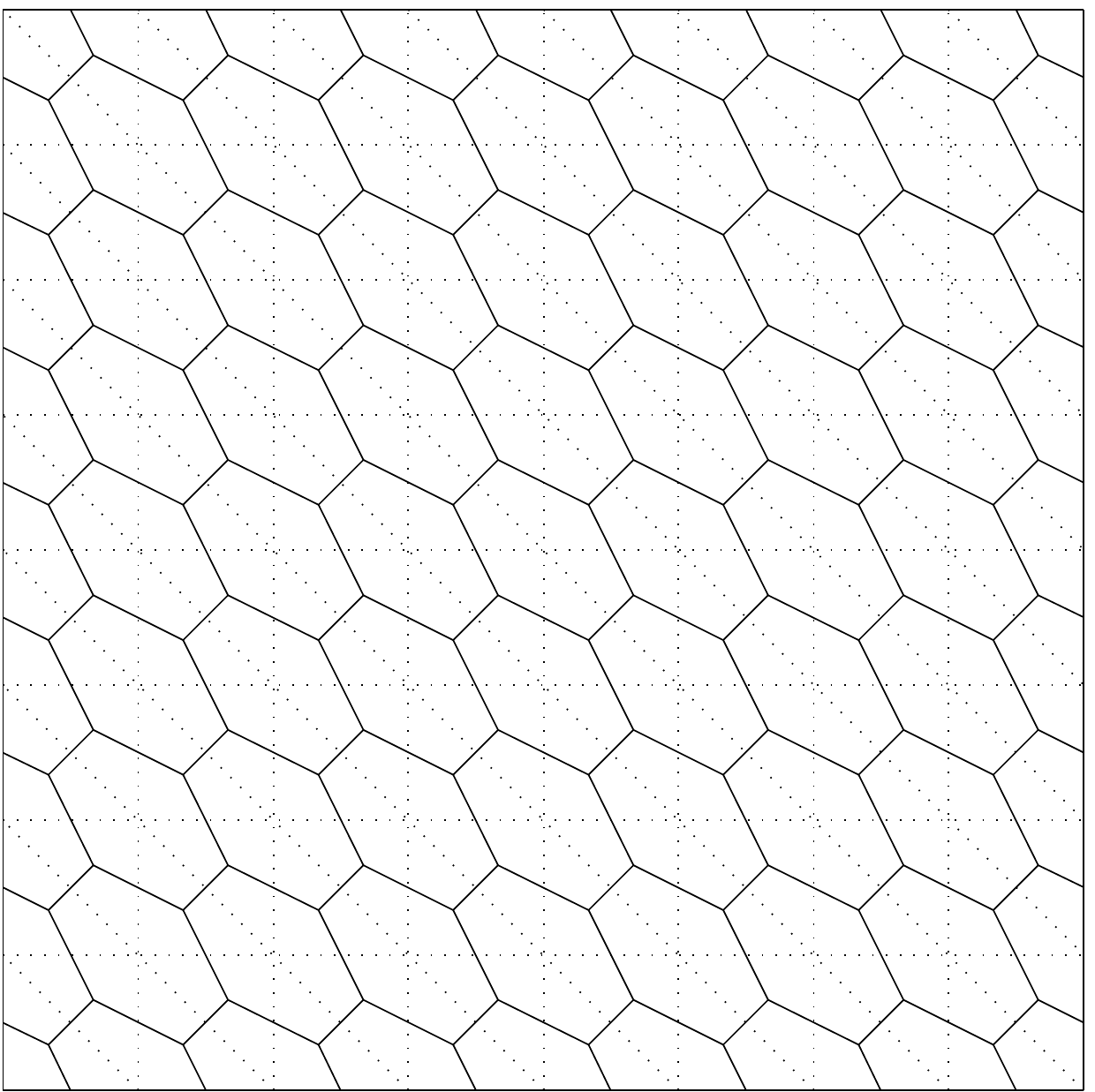}\;
\includegraphics[width=3.5cm]{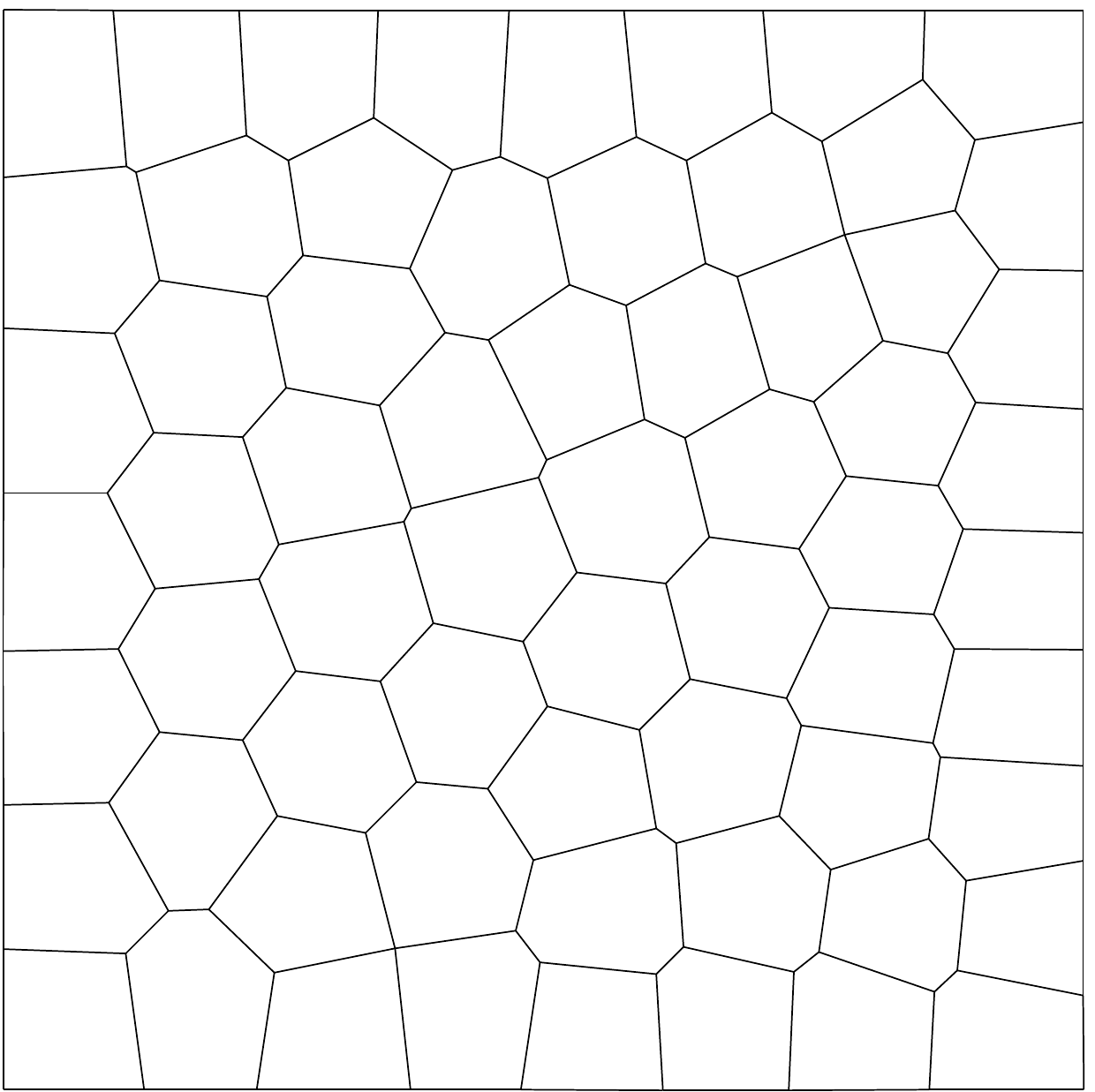}
\caption{Meshes of size $8\times 8$. (1) A quadrilateral mesh. (2) A
hexagonal mesh, with mostly hexagons and a few pentagons and
quadrilaterals. It is generated as the dual mesh of an $8\times 8$
uniform triangular mesh, as shown in dotted lines. (3) Centroidal
Voronoi tessellation consisting of $8\times 8$ cells (see
\cite{Du99} and references therein).} \label{fig:meshes}
\end{center}
\end{figure}

Consider the Poisson's equation on $(0,1)\times(0,1)$ with Dirichlet
boundary condition. We test this problem on three different types of
meshes, as shown in Figure \ref{fig:meshes}. Wachspress coordinates
are used to define $\lambda_i$. The example problem is solved on a
sequence of meshes, using the mixed finite element method with
$\chi(\cM\Lambda^1(T))$-$\bbR$ discretization. Denote by $\vp$ and
$u$ the exact flux and the exact primal solution, while by $\vp_h$
and $u_h$ the corresponding numerical solutions. We first set the
exact solution to be $u=\sin(\pi x)\sin(\pi y)$, which is smooth.
Numerical results are reported in Tables \ref{tab:numericalquad}-\ref{tab:numericalcvt},
in which the `order' is the value of $r$ in $O(h^r)$ computed using the errors on two consecutive meshes.
From the table we can see that $\|\vp-\vp_h\|_{L^2}$, $\|\div\vp -
\div \vp_h\|_{L^2}$ and $\|u-u_h\|_{L^2}$ have at least $O(h)$
convergence, which agrees well with the theoretical prediction.
We also point out that although the centroidal
Voronoi tessellation in Figure \ref{fig:meshes} appears
to contain very short edges, which may theoretically break the
condition given in \cite{Floater14} for the assumption
(\ref{eq:gradlambdaupperbound}), the numerical results presented in Table \ref{tab:numericalcvt}
seem to be unaffected.

\begin{table}[h]
\begin{center}
\caption{Example problem with exact solution $u=\sin(\pi x)\sin(\pi
y)$. Errors of the $\chi(\cM\Lambda^1(T))$-$\bbR$ approximation on
quadrilateral meshes as shown in Figure \ref{fig:meshes}.}
\label{tab:numericalquad}
\begin{tabular}{c|cc|cc|cc}
\hline\hline   &
\multicolumn{2}{|c|}{$\|\vp-\vp_h\|_{L^2}$} &
\multicolumn{2}{|c|}{$\|\div\vp - \div \vp_h\|_{L^2}$} &
\multicolumn{2}{|c}{$\|u-u_h\|_{L^2}$} \\
Mesh Size & error & order & error & order & error & order \\ \hline\hline
$4\times 4$     & 5.2843e-1 &      & 3.1580e+0 &      & 1.6184e-1 & \\ \hline
$8\times 8$     & 2.6040e-1 &1.0210& 1.6087e+0 &0.9731& 8.1764e-2 &0.9850 \\ \hline
$16\times 16$   & 1.2971e-1 &1.0054& 8.0813e-1 &0.9932& 4.0974e-2 &0.9968 \\ \hline
$32\times 32$   & 6.4810e-2 &1.0010& 4.0454e-1 &0.9983& 2.0498e-2 &0.9992 \\ \hline
$64\times 64$   & 3.2405e-2 &1.0000& 2.0233e-1 &0.9996& 1.0251e-2 &0.9997 \\ \hline
$128\times 128$ & 1.6204e-2 &0.9999& 1.0117e-1 &0.9999& 5.1255e-3 &1.0000 \\ \hline
$256\times 256$ & 8.1023e-3 &0.9999& 5.0587e-2 &0.9999& 2.5628e-3 &1.0000 \\ \hline
$512\times 512$ & 4.0513e-3 &0.9999& 2.5293e-2 &1.0000& 1.2814e-3 &1.0000 \\ \hline\hline
\end{tabular}
\end{center}
\end{table}

\begin{table}[h]
\begin{center}
\caption{Example problem with exact solution $u=\sin(\pi x)\sin(\pi
y)$. Errors of the $\chi(\cM\Lambda^1(T))$-$\bbR$ approximation on
hexagonal meshes as shown in Figure \ref{fig:meshes}.}
\label{tab:numericalhexa}
\begin{tabular}{c|cc|cc|cc}
\hline\hline   &
\multicolumn{2}{|c|}{$\|\vp-\vp_h\|_{L^2}$} &
\multicolumn{2}{|c|}{$\|\div\vp - \div \vp_h\|_{L^2}$} &
\multicolumn{2}{|c}{$\|u-u_h\|_{L^2}$} \\
Mesh Size & error & order & error & order & error & order \\ \hline\hline
$4\times 4$     & 2.7502e-1 &      & 2.6008e+0 &      & 1.3488e-1 & \\ \hline
$8\times 8$     & 1.0994e-1 &1.3228& 1.4988e+0 &0.7951& 7.6665e-2 &0.8150 \\ \hline
$16\times 16$   & 4.5041e-2 &1.2874& 7.9379e-1 &0.9170& 4.0330e-2 &0.9267 \\ \hline
$32\times 32$   & 2.0013e-2 &1.1703& 4.0721e-1 &0.9630& 2.0646e-2 &0.9660 \\ \hline
$64\times 64$   & 9.4150e-3 &1.0879& 2.0608e-1 &0.9826& 1.0442e-2 &0.9835 \\ \hline
$128\times 128$ & 4.5673e-3 &1.0436& 1.0365e-1 &0.9915& 5.2510e-3 &0.9917 \\ \hline
$256\times 256$ & 2.2498e-3 &1.0215& 5.1973e-2 &0.9959& 2.6330e-3 &0.9959 \\ \hline
$512\times 512$ & 1.1166e-3 &1.0107& 2.6023e-2 &0.9980& 1.3184e-3 &0.9979 \\ \hline\hline
\end{tabular}
\end{center}
\end{table}

\begin{table}[h]
\begin{center}
\caption{Example problem with exact solution $u=\sin(\pi x)\sin(\pi
y)$. Errors of the $\chi(\cM\Lambda^1(T))$-$\bbR$ approximation on
centroidal Voronoi tessellations as shown in Figure \ref{fig:meshes}.}
\label{tab:numericalcvt}
\begin{tabular}{c|cc|cc|cc}
\hline\hline   &
\multicolumn{2}{|c|}{$\|\vp-\vp_h\|_{L^2}$} &
\multicolumn{2}{|c|}{$\|\div\vp - \div \vp_h\|_{L^2}$} &
\multicolumn{2}{|c}{$\|u-u_h\|_{L^2}$} \\
Mesh Size & error & order & error & order & error & order \\ \hline\hline
$4\times 4$     & 4.5335e-1 &      & 3.1186e+0 &      & 1.6102e-1 & \\ \hline
$8\times 8$     & 1.8368e-1 &1.3034& 1.5915e+0 &0.9705& 8.1220e-2 &0.9873 \\ \hline
$16\times 16$   & 7.4684e-2 &1.2983& 7.7831e-1 &1.0320& 3.9513e-2 &1.0395 \\ \hline
$32\times 32$   & 2.9515e-2 &1.3394& 3.9116e-1 &0.9926& 1.9829e-2 &0.9947 \\ \hline
$64\times 64$   & 1.3361e-2 &1.1434& 1.9703e-1 &0.9893& 9.9831e-3 &0.9901 \\ \hline
$128\times 128$ & 6.3094e-3 &1.0825& 9.7955e-2 &1.0082& 4.9627e-3 &1.0084 \\ \hline
$256\times 256$ & 3.0048e-3 &1.0702& 4.8807e-2 &1.0050& 2.4726e-3 &1.0051 \\ \hline\hline
\end{tabular}
\end{center}
\end{table}

It would be interesting to compare the numerical results on quadrilateral meshes
given in Table \ref{tab:numericalquad}, with the numerical results of
the lowest order Raviart-Thomas element presented in \cite{Arnold05}.
The Raviart-Thomas element can be extended to convex quadrilaterals via
the Piola transform associated to a bilinear isomorphism, but with a
degeneration of approximation rate in $\|\div(\vp-\vp_h)\|_{L^2}$
(see \cite{Arnold05}). It is not hard to check that, on
quadrilaterals that are not parallelograms, the space
$\chi\left(\mathcal{M}\Lambda^1(T)\right)$ is indeed different from
the polynomial-valued, lowest-order Raviart-Thomas element via Piola transform, because
in this case $\chi\left(\mathcal{M}\Lambda^1(T)\right)$ consists of
rational functions. Therefore, the
$\chi\left(\mathcal{M}\Lambda^1(T)\right)$-$\bbR$ discretization
will still provide optimal $O(h)$ convergence rate in
$\|\div(\vp-\vp_h)\|_{L^2}$, as shown in Table \ref{tab:numericalquad}.
In comparison, numerical results given in
\cite{Arnold05}, using the lowest order Raviart-Thomas element via
Piola transform, does not convergence in $\|\div\vp -
\div \vp_h\|_{L^2}$ when the mesh consists of general quadrilaterals.

\begin{table}[h]
\begin{center}
\caption{Example problem with exact solution $u\in H^{3/2}$.
Errors of the $\chi(\cM\Lambda^1(T))$-$\bbR$ approximation on
quadrilateral meshes as shown in Figure \ref{fig:meshes}.}
\label{tab:numerical2}
\begin{tabular}{c|cc|cc}
\hline\hline   &
\multicolumn{2}{|c|}{$\|\vp-\vp_h\|_{L^2}$} &
\multicolumn{2}{|c}{$\|u-u_h\|_{L^2}$} \\
Mesh Size & error & order & error & order \\ \hline\hline
$4\times 4$     &  9.1202e-2 &      & 4.6653e-2 & \\ \hline
$8\times 8$     &  6.5317e-2 &0.4816& 2.3850e-2 &0.9680 \\ \hline
$16\times 16$   &  4.6480e-2 &0.4909& 1.2045e-2 &0.9856 \\ \hline
$32\times 32$   &  3.2970e-2 &0.4955& 6.0512e-3 &0.9931 \\ \hline
$64\times 64$   &  2.3350e-2 &0.4977& 3.0326e-3 &0.9967 \\ \hline
$128\times 128$ &  1.6524e-2 &0.4989& 1.5181e-3 &0.9983 \\ \hline
$256\times 256$ &  1.1689e-2 &0.4994& 7.5946e-4 &0.9992 \\ \hline
$512\times 512$ &  8.2669e-3 &0.4997& 3.7984e-4 &0.9996 \\\hline\hline
\end{tabular}
\end{center}
\end{table}

We also test a second example problem, under the same settings but with
exact solution $u = \sqrt{\frac{1}{2}(\rho-x)} - \frac{1}{4}\rho^2$, where
$\rho$ is the radius in polar coordinates. One can easily verify that
$-\Delta u = 1$ on $(0,1)\times(0,1)$, and moreover, $u\in
H^{3/2}((0,1)^2)$. Numerical results for the second example problem
using the quadrilateral meshes
are reported in Table \ref{tab:numerical2}. Note that $\|\div\vp -
\div \vp_h\|_{L^2}$ is not included since for this test problem, one
has $\div\vp = \div \vp_h \equiv -1$. From the table, we observe
that $\|\vp-\vp_h\|_{L^2}$ is of approximately $O(h^{1/2})$, which
is reasonable because $\vp\in (H^{1/2})^2$, while $\|u-u_h\|_{L^2}
\approx O(h)$ because $u$ is in $H^{3/2}$.

\section{Construction in 3D} \label{sec:3D}

\subsection{Definitions and properties}
Let $T$ be a convex polyhedron satisfying Assumptions 1-2. Denote by
$\vv_i$, $i=1,\ldots,n$, the vertices of $T$. Then, for each pair of
indices $\{i,j\}$, $1\le i,j\le n$, we have the Whitney $1$-form
$W_{ij}$. Similarly, for each triplet of indices $\{i,j,k\}$, $1\le
i,j,k\le n$, we have the Whitney $2$-form $W_{ijk}$. It is not hard
to see that Whitney forms have the following properties:
$$
\begin{aligned}
W_{ii} &= 0,\qquad W_{ij} = -W_{ji}, \\
W_{ijk} &= 0, \quad\textrm{if at least two of }i,j,k\textrm{ are
identical}, \\
W_{ijk} &= W_{jki} = W_{kij} = - W_{ikj} = -W_{jki} = -W_{kji}.
\end{aligned}
$$
Moreover, using the definition of Whitney forms, Equation
(\ref{eq:GBCLinearPrecision}) and elementary vector calculus
identities, one has
\begin{equation} \label{eq:curlWij}
\curl\, W_{ij} = 2 \nabla \lambda_i \times \nabla\lambda_j = 2
\sum_{k=1}^n W_{ijk}.
\end{equation}
We also state a result from \cite{Gillette14}. Denote by
$\vtau_{ij}=\vv_j-\vv_i$ for all $1\le i,j\le n$. For any constant
vector $\va\in \bbR^3$, one has
\begin{equation} \label{eq:WijLinearCombinations}
\begin{aligned}
\frac{1}{2} \sum_{1\le i, j\le n} (\va\cdot\vtau_{ij}) W_{ij} &=
\sum_{i<j} (\va\cdot\vtau_{ij}) W_{ij} = \va, \\
\frac{1}{2} \sum_{1\le i, j\le n} ((\va\times\vv_i)\cdot\vtau_{ij})
W_{ij} &= \sum_{i<j} ((\va\times\vv_i)\cdot\vtau_{ij}) W_{ij} \\
&= \sum_{i<j} ((\va\times\vv_i)\cdot\vv_j) W_{ij} = \va\times \vx.
\end{aligned}
\end{equation}

The reason that Whitney forms are so important in the construction
of $H(\curl)$ and $H(\div)$ spaces is that, they naturally satisfy
certain conditions on edges/faces of $T$. Before summarizing these
in lemmas, we first need to clarify the concept of `edges'. Denote
by $e_{ij}$ the directed line segment pointing from $\vv_i$ to
$\vv_j$, and by $|e_{ij}|$ its length . Notice that $e_{ij}$ may not
be a natural edge of polyhedron $T$. Indeed, we classify all
$e_{ij}$ into three disjoint categories:
\begin{enumerate}
\item $\cE$ is the set of all $e_{ij}$ that coincides with a natural
edge of $T$;
\item $\cE_F$ is the set of all $e_{ij}$ lying on $\partial T$ but
not in $\cE$;
\item $\cE_I$ is the set of all $e_{ij}$ in the interior of $T$,
i.e., not lying on $\partial T$.
\end{enumerate}
An illustration of these categories is given in Figure
\ref{fig:threeEdgeTypes}. We point out that each category actually
contains both $e_{ij}$ and $e_{ji}$, for a given pair of indices $i$
and $j$. The union of all three categories covers all $e_{ij}$, for
$1\le i,j\le n$. Notice that $\cE_F$ and $\cE_I$ can be empty for
certain polyhedra. On each $e_{ij}$, denote by
 $\vt_{ij}$ the unit tangential
vector pointing from $\vv_i$ to $\vv_j$. We emphasize that only the
$e_{ij}\in\cE$ will be called an `edge' of $T$, while the others are
just called `directed line segments'.

\begin{figure}[h]
\begin{center}
\includegraphics[width=8cm]{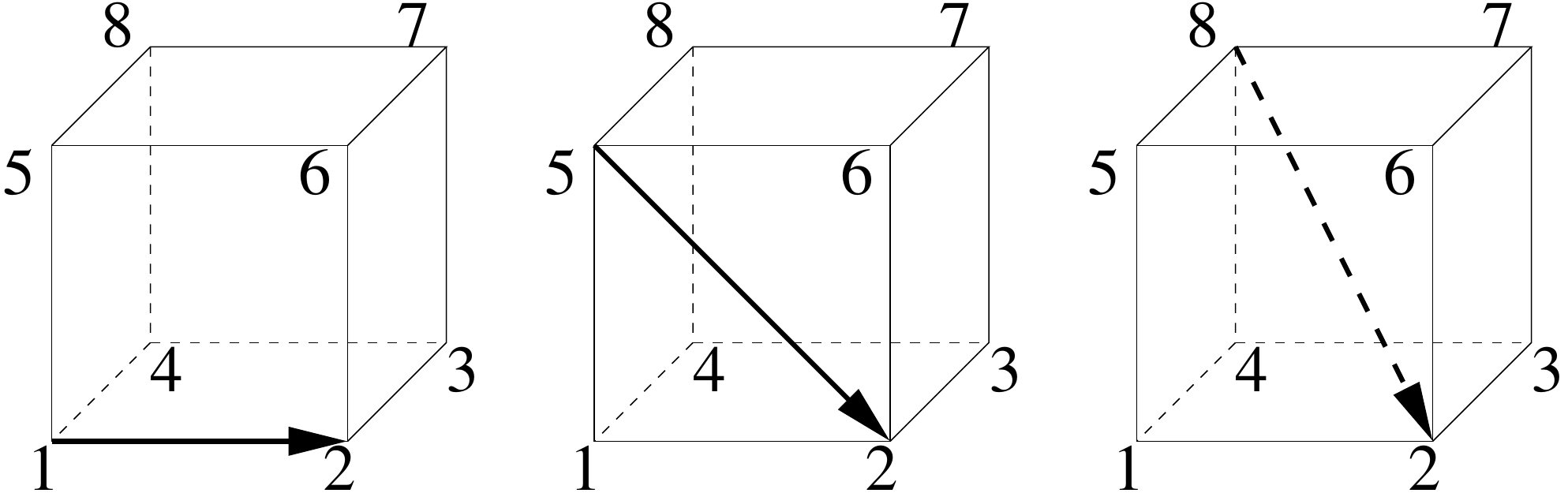}
\caption{Illustration of three categories: $e_{12}\in \cE$,
$e_{52}\in \cE_F$, $e_{82}\in \cE_I$.} \label{fig:threeEdgeTypes}
\end{center}
\end{figure}

\begin{lem} \label{lem:WijOnEdges}
Let $e_{kl}\in \cE$. Then for all $1\le i,j\le n$, one has
$$
W_{ij}\cdot \vt_{kl}|_{e_{kl}} = \begin{cases}
\frac{1}{|e_{ij}|}\quad
&\textrm{if }e_{ij} = e_{kl}, \\
-\frac{1}{|e_{ij}|}\quad
&\textrm{if }e_{ij} = e_{lk}, \\
0 &\textrm{otherwise}.
\end{cases}
$$
\end{lem}
\begin{proof}
The proof follows immediately from the definitions of $\lambda_i$,
$W_{ij}$ and Assumption 1, which states that $\lambda_i$ is linear
on all $e_{kl}\in \cE$.
\end{proof}

\begin{rem}
On $e_{kl}\in\cE_F$ or $\cE_I$, we do not have results similar to
Lemma \ref{lem:WijOnEdges}, since $\lambda_i$ may not even be linear
on $e_{kl}$.
\end{rem}

Next we define another important form on each $e_{ij}\in\cE$. Denote
by $\cF_{ij}$ the set of two faces of polyhedron $T$ that share the
edge $e_{ij}$, and by $\cV_{ij}$ the set of all vertices on
$\cF_{ij}$. For a fixed index $1\le i\le n$, note that any
$\vtau_{ik}$, for $e_{ik}\in\cE_I$ can be written as a linear
combination of all $\vtau_{ij}$, for $e_{ij}\in\cE$. Such a linear
combination is not uniquely defined if vertex $\vv_i$ is connected
to more than 3 edges of the polyhedron. Nevertheless, we can always
fix a linear combination for each vertex $\vv_i$, and denote this
chosen one by
\begin{equation} \label{eq:Cikij}
\vtau_{ik} = \sum_{j,\,e_{ij}\in\cE} C^{ik}_{ij} \vtau_{ij}.
\end{equation}
Now, define
$$
\begin{aligned}
\tW_{ij} &= W_{ij} + \frac{1}{2} \left(\sum_{\vv_k\in\cV_{ij},\,
e_{ik}\in\cE_F} W_{ik} - \sum_{\vv_k\in\cV_{ij},\,e_{jk}\in\cE_F}
W_{jk} \right) \\
&\qquad + \frac{1}{2} \left(\sum_{k,\, e_{ik}\in\cE_I} C^{ik}_{ij}
W_{ik} - \sum_{k,\, e_{jk}\in\cE_I} C^{jk}_{ji} W_{jk} \right).
\end{aligned}
$$
In the above, one may view $W_{ij} + \frac{1}{2}
\left(\sum_{\vv_k\in\cV_{ij},\, e_{ik}\in\cE_F} W_{ik} -
\sum_{\vv_k\in\cV_{ij},\,e_{jk}\in\cE_F} W_{jk} \right)$ as the
`surface' component of $\tW_{ij}$ and $\frac{1}{2} \left(\sum_{k,\,
e_{ik}\in\cE_I} C^{ik}_{ij} W_{ik} - \sum_{k,\, e_{jk}\in\cE_I}
C^{jk}_{ji} W_{jk} \right)$ as the `interior' component of
$\tW_{ij}$. An illustration of the surface component of $\tW_{ij}$,
which can also be written as $W_{ij} + \frac{1}{2}
\left(\sum_{\vv_k\in\cV_{ij},\, e_{ik}\in\cE_F} W_{ik} +
\sum_{\vv_k\in\cV_{ij},\,e_{kj}\in\cE_F} W_{kj} \right)$, is given
in Figure \ref{fig:tildeWij}. Note that if both faces sharing
$e_{ij}$ are triangles, the surface component of $\tW_{ij}$ is just
$W_{ij}$.

\begin{figure}[h]
\begin{center}
\includegraphics[width=8cm]{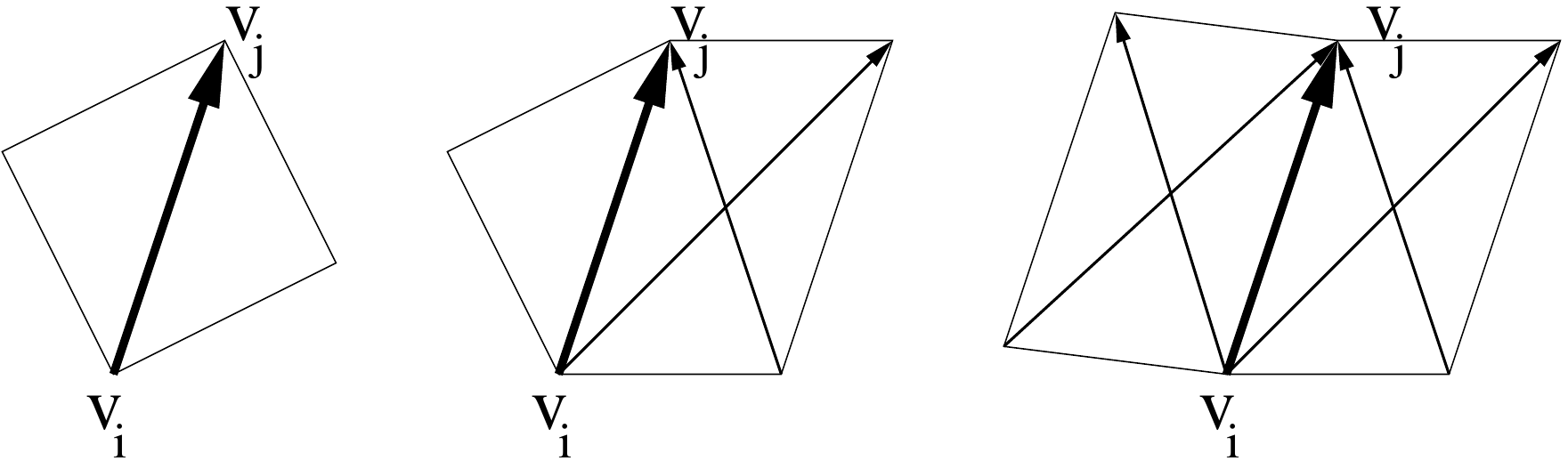}
\caption{Illustration of surface component of $\tW_{ij}$ when
$e_{ij}\in\cE$ is shared by two faces of $T$ which are: (1) two
triangles; (2) one triangle and one parallelogram; (3) two
parallelograms. Here we conveniently use thick arrow to denote
$W_{kl}$ and thin arrow to denote $\frac{1}{2} W_{kl}$ on any
$e_{kl}$.} \label{fig:tildeWij}
\end{center}
\end{figure}

The vector function $\tW_{ij}$ has many nice properties. First, it
is obvious that $\tW_{ij} = -\tW_{ji}$. Now, let us fix a direction
for each edge of $T$. The collection of all edges in $\cE$, with the
prefixed direction, is denoted by $\cE^+$. Similarly, one may denote
the collection of all edges in $\cE$ with direction opposite to the
prefixed one as $\cE^-$. The two sets $\cE^+$ and $\cE^-$ contain
the same edges, but with opposite directions. For any two edges
$e_{ij}$ and $e_{kl}$ in $\cE^+$, denote by $\delta_{e_{ij},e_{kl}}$
the Kronecker delta whose value is $1$ if $e_{ij}=e_{kl}$ and $0$
otherwise. Then, we have the following lemmas:

\begin{lem} \label{lem:twijOnEdges}
The set $\{\tW_{ij},\textrm{ for }e_{ij}\in \cE^+\}$ satisfy
$\tW_{ij}\cdot \vt_{kl}|_{e_{kl}} =
\frac{1}{|e_{ij}|}\delta_{e_{ij},e_{kl}}$ for all $e_{kl}\in\cE^+$,
and hence is linearly independent.
\end{lem}
\begin{proof}
This follows immediately from the definition of $\tW_{ij}$, Lemma
\ref{lem:WijOnEdges}, and the fact that $\sum_{e_{ij}\in\cE^+}
c_{ij}\tW_{ij} = \mathbf{0}$ implies that $c_{kl} =
|e_{kl}|\left(\sum_{e_{ij}\in\cE^+} c_{ij}\tW_{ij}\right)\cdot
\vt_{kl}|_{e_{kl}} = 0$ for all $e_{kl}\in\cE^+$.
\end{proof}

\begin{lem} \label{lem:tWcontainP1-}
It holds that $\cP_1^-\Lambda^1(T)\subseteq span \{\tW_{ij},\textrm{
for }e_{ij}\in \cE^+\}$.
\end{lem}
\begin{proof}
Let us first point out that $span \{\tW_{ij},\textrm{ for }e_{ij}\in
\cE^+\} = span \{\tW_{ij},\textrm{ for }e_{ij}\in \cE\}$. By the
definitions of $\tW_{ij}$ and $C^{ik}_{ij}$, Equation
(\ref{eq:WijLinearCombinations}), Assumption 2, and the fact that
$\vv_i\times\vtau_{ij} = -\vv_j\times\vtau_{ji}$, for any
$\va\in\bbR^3$ one has
$$
\begin{aligned}
\sum_{e_{ij}\in\cE} &((\va\times \vv_i) \cdot\vtau_{ij}) \tW_{ij} =
\sum_{e_{ij}\in\cE} ((\va\times \vv_i) \cdot\vtau_{ij}) W_{ij} \\
& + \frac{1}{2} \sum_{i=1}^n \sum_{k,\, e_{ik}\in\cE_F} ((\va\times
\vv_i) \cdot\vtau_{ik}) W_{ik} + \frac{1}{2} \sum_{j=1}^n
\sum_{k,\,e_{jk}\in\cE_F} ((\va\times
\vv_j) \cdot\vtau_{jk}) W_{jk}  \\
& + \frac{1}{2} \sum_{i=1}^n \sum_{k,\, e_{ik}\in\cE_I} ((\va\times
\vv_i) \cdot\vtau_{ik}) W_{ik} + \frac{1}{2} \sum_{j=1}^n \sum_{k,\,
e_{jk}\in\cE_I} ((\va\times \vv_j) \cdot\vtau_{jk}) W_{jk}\\
&= \sum_{i=1}^n \left( \sum_{k=1}^n ((\va\times \vv_i)
\cdot\vtau_{ik}) W_{ik} \right) \\
&= 2\va\times\vx.
\end{aligned}
$$
This indicates that $\va\times\vx\in span \{\tW_{ij},\textrm{ for
}e_{ij}\in \cE^+\}$. Similarly, one can prove that for any
$\vb\in\bbR^3$,
$$
\sum_{e_{ij}\in\cE} (\vb\cdot\vtau_{ij}) \tW_{ij} = 2\vb.
$$
Recall that $\cP_1^-\Lambda^1(T) = span \{\va\times \vx + \vb,
\textrm{ for all }\va,\vb\in\bbR^3\}$. This completes the proof of
the lemma.
\end{proof}

Denote by $\cF$ the set of all faces of $T$, and by $\vn_f$ the unit
outward normal vector on $f\in\cF$ with respect to $T$. For each
$f\in\cF$, denote by $|f|$ its area and by $\partial f$ the oriented
boundary of $f$ such that its orientation satisfies the right-hand
rule with $\vn_f$. If $e_{ij}$ lies on $\partial f$ and has the same
direction as the orientation of $\partial f$, we say $e_{ij}\in
\partial f$. If $e_{ij}$ lies on $\partial f$ and has the opposite
direction as the orientation of $\partial f$, we say $e_{ij}\in
-\partial f$.

\begin{lem} \label{lem:curltw}
Let $f\in \cF$ and $e_{ij}\in\cE$, then one has
$$
\curl \tW_{ij} \cdot \vn_f |_f = \begin{cases} \frac{1}{|f|}\quad &
\textrm{if }e_{ij}\in\partial f, \\
- \frac{1}{|f|}\quad & \textrm{if }e_{ij}\in -\partial
f,\\
0 \quad&\textrm{otherwise}.\end{cases}
$$
\end{lem}
\begin{proof}
Notice that for any $f\in \cF$ and $1\le i,j\le n$, by Equation
(\ref{eq:curlWij}), one has
$$
\curl W_{ij} \cdot \vn_f|_f = 2(\nabla\lambda_i \times
\nabla\lambda_j)\cdot \vn_f|_f = 2(\nabla_f\lambda_i \times
\nabla_f\lambda_j)\cdot \vn_f|_f,
$$
where $\nabla_f \lambda_i|_f$ denotes the tangential component of
$\nabla \lambda_i$ on $f$. By Assumption 1, $\nabla_f \lambda_i|_f$
is non-zero only if $\vv_i$ is a vertex on face $f$. It is then
clear that $\curl W_{ij} \cdot \vn_f|_f$ is non-zero only when both
$\vv_i$ and $\vv_j$ are vertices of face $f$. Consequently, $\curl
\tW_{ij} \cdot \vn_f|_f$ is non-zero only when $e_{ij}\in\partial f$
or $-\partial f$.

For $f\in\cF$, denote by $\cV(f)$ the set of vertices on face $f$.
Without loss of generality, assume $f$ lies on the $xy$-plane with
outward normal $\vn_f=[0,0,1]^t$, and denote by $\lambda_k^{(2)}$,
for all $\vv_k\in\cV(f)$, the 2-dimensional barycentric coordinates
on polygon $f$. By Assumption 1, the 3D coordinate $\lambda_k$,
where $\vv_k\in\cV(f)$, degenerates to $\lambda^{(2)}_k$ on $f$.
Consequently, $\nabla_f \lambda_k$
is equal to $\begin{bmatrix}\nabla^{(2)}\lambda^{(2)}_k\\
0\end{bmatrix}$, where $\nabla^{(2)}$ stands for the 2-dimensional
gradient on the $xy$-plane. Note we have for all $e_{kl}\in\partial
f$ that
$$
(\nabla_f\lambda_k \times \nabla_f\lambda_l)\cdot \vn_f|_f =
\bigg( \begin{bmatrix}\nabla^{(2)}\lambda^{(2)}_k\\
0\end{bmatrix} \times \begin{bmatrix}\nabla^{(2)}\lambda^{(2)}_l\\
0\end{bmatrix} \bigg) \cdot \begin{bmatrix}0\\0\\1\end{bmatrix} =
det [\nabla^{(2)}\lambda^{(2)}_k\; \nabla^{(2)}\lambda^{(2)}_l].
$$

Consider the case $e_{ij}\in \partial f$. When $f$ is a triangle, it
is clear that by Lemma \ref{lem:triangleCross}
$$
\curl \tW_{ij} \cdot \vn_f|_f = \curl W_{ij} \cdot \vn_f|_f =
2(\nabla_f\lambda_i \times \nabla_f\lambda_j)\cdot \vn_f|_f =
\frac{1}{|f|}.
$$
When $f$ is a parallelogram, denote by $\vv_{i}$, $\vv_{j}$
$\vv_{k}$ and $\vv_l$ the vertices on $f$ such that $\partial f =
\{e_{ij},e_{jk},e_{kl},e_{li}\}$. Then by the definition of
$\tW_{ij}$ and Lemma \ref{lem:parallelogramCross},
$$
\begin{aligned}
(\curl \,& \tW_{ij}) \cdot \vn_f|_{f} = \curl (W_{ij} + \frac{1}{2}
W_{ik} + \frac{1}{2} W_{lj} )\cdot \vn_f|_{f} \\
&= \curl \left(\frac{1}{2}(W_{ii} + W_{ij} + W_{ik}+W_{il})  +
\frac{1}{2}
(W_{ij} +W_{jj}+W_{kj}+W_{lj} )\right) \cdot \vn_f|_{f} \\
&\qquad - \frac{1}{2} \curl (W_{il} + W_{kj}) \cdot \vn_f|_{f} \\
&= (\nabla_{f} \lambda_i \times \sum_{s\in\{i,j,k,l\}}\nabla_{f}
\lambda_s)\cdot \vn_f|_{f} +  \sum_{s\in\{i,j,k,l\}}\nabla_{f}
\lambda_s \times \nabla_{f}
\lambda_j)\cdot \vn_f|_{f}    \\
&\qquad +  ( \nabla_{f} \lambda_l \times \nabla_{f} \lambda_i +
 \nabla_{f} \lambda_j \times \nabla_{f} \lambda_k )\cdot \vn_f|_{f}  \\
 &=0+0 + \frac{1}{|f|}.
\end{aligned}
$$
In the above we have used $W_{ii}=W_{jj}=0$ and
$\sum_{s\in\{i,j,k,l\}} \nabla^{(2)} \lambda_s^{(2)} =\nabla^{(2)} 1
= \mathbf{0}$.

 For $e_{ij}\in -\partial f$, one just needs to change the sign.
This completes the proof of the lemma.
\end{proof}

Finally, on each $f\in\cF$, define
$$
\tW_f = \sum_{e_{ij}\in\partial f} \tW_{ij}.
$$
By the definition and Lemma \ref{lem:curltw}, we clearly have
\begin{equation} \label{eq:curltWf1}
(\curl \, \tW_f) \cdot \vn_f|_f =
\begin{cases}\frac{3}{|f|} \quad &\textrm{if }f\textrm{ is a triangle}, \\ \frac{4}{|f|} \quad &\textrm{if }f\textrm{ is a parallelogram}.\end{cases}
\end{equation}
Moreover, let $f'\in\cF$ be another face of $T$ that is different
from $f$, then
\begin{equation} \label{eq:curltWf2}
(\curl \, \tW_f) \cdot \vn_{f'}|_{f'} = \begin{cases}-\frac{1}{|f'|}
\quad &\textrm{if }f,f'\textrm{ share an edge}, \\ 0 \quad
&\textrm{if }f,f'\textrm{ do not share edge}.
\end{cases}
\end{equation}

\subsection{Discrete space and Basis function}
Now we are able to construct spaces $\cM\Lambda^1(T)$ and
$\cM\Lambda^2(T)$. It is very tempting to use $\tW_{ij}$, for all
$e_{ij}\in\cE^+$ as a set of basis for the $H(\curl)$ finite element
space $\cM\Lambda^1(T)$. However, the biggest problem of doing so is
that, we are not sure whether $\nabla \lambda_i\in
span\{\tW_{ij},\textrm{ for }e_{ij}\in \cE^+\}$ or not, and thus can
not ensure the $\nabla \cM\Lambda^0(T)\subset \cM\Lambda^1(T)$ part
in the sequence (\ref{eq:minimalComplex3D}).

To ensure the exactness of sequence (\ref{eq:minimalComplex3D}),
similar to the 2D case, we will try
$$
\begin{aligned}
\cM\Lambda^1(T) &= \nabla \cM\Lambda^0(T) \oplus \cH = span\{\nabla
\lambda_i,\,i=1,\ldots,n\} \oplus \cH, \\
\cM\Lambda^2(T) &= \curl \,\cH \oplus (\div^{\dagger}) \bbR,
\end{aligned}
$$
where $\cH$ is a space orthogonal to $span\{\nabla
\lambda_i,\,i=1,\ldots,n\}$. Again, in practice, it is very hard to
construct orthogonal basis. Thus we relax the orthogonality
requirement a little bit and replace $\oplus$ by $+$. Similar to
(\ref{eq:HdivSpace2D}), we construct the following:
\begin{align}
\cM\Lambda^1(T) &= span\{\nabla \lambda_i,\,i=1,\ldots,n\} +
span\{\tW_f,\, f\in\cF\}, \label{eq:HcurlSpace3D} \\
\cM\Lambda^2(T) &= \curl\, span\{\tW_f,\, f\in\cF\} + span\{\vx -
\vx_*\} \label{eq:HdivSpace3D} \\ &= span\{\curl\, \tW_f,\,
f\in\cF\} + span\{\vx - \vx_*\}, \nonumber
\end{align}
where $\vx_*$ is a chosen point inside $T$. Of course this is just
the construction. We still need to show that
(\ref{eq:minimalComplex3D}) is exact under this construction.

By definition, we have $\bbR\in \nabla\cM\Lambda^0(T)$,
$\nabla\cM\Lambda^0(T)\subset \cM\Lambda^1(T)$, $\curl
\cM\Lambda^1(T) \subset \cM\Lambda^2(T)$ and $\div \cM\Lambda^2(T) =
\bbR$. Moreover, it is clear that $\curl \cM\Lambda^1(T)\cap
span\{\vx - \vx_*\} = \{\mathbf{0}\}$. These establish the exactness
at the $\cM\Lambda^0(T)$ and the $\cM\Lambda^2(T)$ nodes. To show
that (\ref{eq:minimalComplex3D}) is exact at the $\cM\Lambda^1(T)$
node, we only need to prove that no none-zero vector in
$span\{\tW_f,\, f\in\cF\}$ is curl free. This can indeed be done by
counting dimensions, i.e., we will prove that the dimensions of
$\cM\Lambda^1(T)$ and $\cM\Lambda^2(T)$ are exactly $\#E$ and $\#F$,
as indicated in (\ref{eq:minimalComplex3D}). These dimensions are
computed by explicitly constructing basis functions, as shown in the
following two lemmas. We postpone the proof of these two lemmas to
Appendix \ref{appendix}.

\begin{lem} \label{lem:Hdivbasis3D}
There exists a computable basis $\{\vq_f, \textrm{ for } f\in\cF\}$
for $\cM\Lambda^2(T)$ defined in (\ref{eq:HdivSpace3D}), such that
on each $f'\in\cF$,
$$
\vq_f\cdot\vn_{f'}|_{f'} = \begin{cases} 1\quad &\textrm{if }f=f', \\
0&\textrm{otherwise}.
\end{cases}
$$
Therefore the dimension of $\cM\Lambda^2(T)$ is equal to the number
of faces of $T$.
\end{lem}

\begin{lem} \label{lem:Hcurlbasis3D}
There exists a computable basis $\{\vp_{e}, \textrm{ for }
e\in\cE^+\}$ for $\cM\Lambda^1(T)$ defined in
(\ref{eq:HcurlSpace3D}), such that on each $e'\in \cE^+$,
$$
\vp_{e} \cdot \vt_{e'}|_{e'} = \begin{cases} 1\quad &\textrm{if }e=e', \\
0&\textrm{otherwise}.
\end{cases}
$$
Therefore the dimension of $\cM\Lambda^1(T)$ is equal to the number
of edges of $T$.
\end{lem}

\begin{rem}
By the definitions of $\cM\Lambda^1(T)$ and $\cM\Lambda^2(T)$,
lemmas \ref{lem:Hdivbasis3D}-\ref{lem:Hcurlbasis3D}, and by counting
the dimensions, we know that (\ref{eq:minimalComplex3D}) is an exact
sequence.
\end{rem}

\begin{rem}
Lemma \ref{lem:Hdivbasis3D} indicates that for all $\vq\in
\cM\Lambda^2(T)$, $\vq\cdot\vn$ is piecewise constant on the surface
of $T$. Moreover, the normal components on faces of $T$ form a
unisolvent set of degrees of freedom for $\cM\Lambda^2(T)$, which
allows one to build $H(\div)$ conforming finite element space using
$\cM\Lambda^2(T)$.
\end{rem}

\begin{rem}
Similarly, Lemma \ref{lem:Hcurlbasis3D} indicates that for all
$\vq\in \cM\Lambda^1(T)$, $\vq\cdot\vt$ is piecewise constant on the
skeleton of $T$, i.e., the collection of all edges in $\cE$.
Moreover, the tangential components on edges of $T$ form a
unisolvent set of degrees of freedom for $\cM\Lambda^1(T)$. However,
this is not enough for building $H(\curl)$ conforming finite element
space, as $H(\curl)$ conforming requires the tangential components
on all faces, not only on edges, to be continuous across elements.
\end{rem}

Next, we show that the basis $\vp_e$ also provides tangential
continuity across faces. For each $\vp\in \cM\Lambda^1(T)$, its
value on a face $f\in\cF$ can be split into two orthogonal parts
$$
\vp|_f = \cT_f(\vp) + \cN_f(\vp),
$$
where $\cT_f(\vp)$ and $\cN_f(\vp)$ are the vector projections of
$\vp|_f$ onto $f$ and its normal direction, respectively. We also
denote by $\cT_{\partial T}(\vp)$ the patching of $\cT_f(\vp)$ over
all $f\in\cF$.

By the definition of $\cM\Lambda^1(T)$ and $\tW_f$, it is clear that
$$
\cM\Lambda^1(T) \subseteq span\{\nabla \lambda_i,\,i=1,\ldots,n\} +
span\{\tW_{ij},\, e_{ij}\in\cE\}.
$$
But in general, we do not know whether $\nabla\lambda_i =
-\sum_{j=1}^n W_{ij}$ is in $span\{\tW_{ij},\, e_{ij}\in\cE\}$ or
not. However, if only considering the tangential component, one has
the following nice property:

\begin{lem} \label{lem:tangentialComponent}
Let $e_{ij}\in \cE^+$ and $\vp_{e_{ij}}$ be the basis function of
$\cM\Lambda^1(T)$ associated with edge $e_{ij}$. Then
$$
\cT_{\partial T} (\vp_{e_{ij}}) = \cT_{\partial T}
(|e_{ij}|\tW_{ij}).
$$
\end{lem}
\begin{proof}
Clearly, $\cT_f(W_{ij})$ is nonzero on $f$ only when both $\vv_i$
and $\vv_j$ lie on $f$. Note that Equation (\ref{eq:Wijsummation})
is still true in 3D. Therefore, one has
$$
\cT_{\partial T} (\nabla \lambda_i) = -\sum_{
\begin{matrix}j\textrm{ such that}
\\e_{ij}\in\cE\cup\cE_F \end{matrix}} \cT_{\partial T} (W_{ij}) = -
\sum_{ \begin{matrix}j\textrm{ such that}
\\e_{ij}\in\cE\end{matrix}} \cT_{\partial T} (\tW_{ij}).
$$
In the above we have used the definition of $\tW_{ij}$ to cancel out
terms on $e_{kl}\in\cE_F$ (if there exists any) that are not
connected to vertex $\vv_i$. Hence $ \cT_{\partial T} (\nabla
\lambda_i) \in span\{\cT_{\partial T}(\tW_{ij}),\, e_{ij}\in\cE\}$,
which together with the definitions of $\tW_f$ and
$\cM\Lambda^1(T)$, further implies that
$$
\begin{aligned}
\cT_{\partial T} (\cM\Lambda^1(T))
&\subseteq span\{\cT_{\partial T}(\tW_{ij}),\, e_{ij}\in\cE\} \\
&= span\{\cT_{\partial T}(\tW_{ij}),\, e_{ij}\in\cE^+\}.
\end{aligned}
$$
Therefore, by Lemma \ref{lem:twijOnEdges} and by comparing the
tangential components on each edge, one must have $ \cT_{\partial T}
(\vp_{e_{ij}}) = \cT_{\partial T}(|e_{ij}|\tW_{ij})$ for all
$e_{ij}\in\cE^+$. This completes the proof of the lemma.
\end{proof}

\begin{rem}
Lemma \ref{lem:tangentialComponent} tells us that the tangential
component of each basis function $\vp_{e_{ij}}$ on $\partial T$ is
completely determined by the tangential component of $\tW_{ij}$. Let
$T$ and $T'$ be two polyhedra sharing a face $f$, and let $e_{ij}$
be an edge of the polygon $f$. Then, by the definition of $\tW_{ij}$
and Assumption 1, we know that $\vp_{e_{ij}}$ has continuous
tangential component across the face $f$. Thus one can build
$H(\curl)$ conforming finite element spaces using $\cM\Lambda^1(T)$.
\end{rem}

\begin{rem} \label{rem:noInternalEdge}
If $\cE_I = \emptyset$, then similar to the proof of Lemma
\ref{lem:tangentialComponent}, one can show $\nabla\lambda_i \in
span\{\tW_{ij},\, e_{ij}\in\cE\}$ and consequently
$\vp_{e_{ij}}=|e_{ij}|\tW_{ij}$. Examples of polyhedra with $\cE_I =
\emptyset$ include tetrahedra, pyramids and triangular prisms, but
not rectangular boxes.
\end{rem}

Next, we briefly show that $\cM\Lambda^k(T)\subseteq
\mathcal{W}\Lambda^k(T)$ for $k=1,2$. By Equation
(\ref{eq:Wijsummation}) and the definition of $\tW_{f}$, one
immediately has $\cM\Lambda^1(T)\subseteq\mathcal{W}\Lambda^1(T)$.
Similarly, by Equation (\ref{eq:curlWij}) and the definition of
$\tW_{f}$, one gets $\curl \cM\Lambda^1(T) \subseteq
\mathcal{W}\Lambda^2(T)$. We also know from Equation
(\ref{eq:WLambdaTincludeP-LambdaT}) that $\vx-\vx_*
\in\mathcal{W}\Lambda^2(T)$. Combining the above with the definition
of $\cM\Lambda^2(T)$ gives $\cM\Lambda^2(T)\subseteq
\mathcal{W}\Lambda^2(T)$.

Finally, to ensure the approximation property of $\cM\Lambda^k(T)$,
for $k=1,2$, we would like to have $\cP_1^-\Lambda^k(T)\subseteq
\cM\Lambda^k(T)$. This is not easy to prove, and so far we do not
even know whether it is in general true or not. Fortunately, we are
able to prove this for two special types of polyhedra:
\begin{description}
\item[Type I] Polyhedra with $\cE_I=\emptyset$;
\item[Type II] Polyhedra with a center $\vx_c$ such that for
each vertex $\vv_i$, $1\le i\le n$, one has
\begin{equation} \label{eq:typeIIPolyhedron}
(\vx_c - \vv_i)\times \sum_{j,\,e_{ij}\in\cE} \vtau_{ij} = 0.
\end{equation}
This is equivalent to say the barycenter of the point set
$\{\vv_j,\textrm{ for all }e_{ij}\in\cE\}$ lies in the line passing
through $\vv_i$ and $\vx_c$.
\end{description}

The proof of the following Lemma will be given in Appendix
\ref{appendix2}.
\begin{lem} \label{lem:ContainingR}
On Type I and II polyhedra, one has $\cP_1^-\Lambda^k(T)\subseteq
\cM\Lambda^k(T)$ for $k=1,2$.
\end{lem}

\begin{rem}
Type I polyhedra include all tetrahedra, pyramids, and triangular
prisms. Type II polyhedra include all parallelepipeds, all regular
$n$-gon based bipyramids, the regular octahedron, the regular
icosahedron, and some Catalan solids.
\end{rem}

\begin{rem}
From Lemma \ref{lem:ContainingR}, we know that for Type I and II
polyhedra, the definition of $\cM\Lambda^2(T)$ is independent of the
choice of $\vx_*$, because $\bbR^3\subset \cP_1^-\Lambda^2(T)$. But
so far we do not know whether the definitions of $\cM\Lambda^1(T)$
and $\cM\Lambda^2(T)$ are independent of the linear combination
given  in Equation (\ref{eq:Cikij}) or not.
\end{rem}

\begin{rem} \label{rem:alternativeConstruction}
One may alternatively define an $H(\curl)$ conforming finite element
$$
\widetilde{\cM\Lambda^1}(T) = span\{\tW_{ij},\textrm{ for
}e_{ij}\in\cE^+\},
$$
which contains $\cP_1^-\Lambda^1(T)$ according to Lemma
\ref{lem:tWcontainP1-} for all polyhedra satisfying Assumptions 1-2
(not restricted to Type I and II polyhedra). Moreover, by Remark
\ref{rem:noInternalEdge}, it is clear that
$\widetilde{\cM\Lambda^1}(T)= \cM\Lambda^1(T)$ on Type I polyhedra.
However, as mentioned in the beginning of this section, in general
we do not know whether the alternative construction fits into a
discrete exact sequence similar to (\ref{eq:minimalComplex3D}) or
not.
\end{rem}


\subsection{Examples}
We show that our construction reproduces known $H(\curl)$ and
$H(\div)$ elements on tetrahedra, rectangular boxes, pyramids, and
triangular prisms. Then, we shall construct elements on a regular
octahedron, which has never been done before.

In the construction, basis functions are computed according to the
proof of lemmas \ref{lem:Hdivbasis3D} and \ref{lem:Hcurlbasis3D},
which is given in Appendix \ref{appendix}. Wachspress coordinates
are used to define $\lambda_i$. The computation can be done using
any computer algebra system. The results are listed below:
\begin{enumerate}
\item On any tetrahedron, there exists a unique set of barycentric coordinates. One can indeed easily prove that
$\cM\Lambda^k(T)=\mathcal{W}\Lambda^k(T) =
\mathcal{P}_1^-\Lambda^k(T)$ for $k=0,1,2$. No computation is
needed.
\item On a rectangular box $(0,h_1)\times (0,h_2)\times (0,h_3)$, by
using the standard tensor product basis:
$$
\begin{aligned}
\lambda_1 &= \frac{(h_1-x)(h_2-y)(h_3-z)}{h_1h_2h_3}, \quad
&\lambda_2 &= \frac{x(h_2-y)(h_3-z)}{h_1h_2h_3}, \\ \lambda_3 &=
\frac{xy(h_3-z)}{h_1h_2h_3},\quad & \lambda_4 &=
\frac{(h_1-x)y(h_3-z)}{h_1h_2h_3}, \\
\lambda_5 &= \frac{((h_1-x)(h_2-y)z}{h_1h_2h_3}, \quad &\lambda_6 &=
\frac{x(h_2-y)z}{h_1h_2h_3}, \\ \lambda_7 &=
\frac{xyz}{h_1h_2h_3},\quad & \lambda_8 &=
\frac{(h_1-x)yz}{h_1h_2h_3},
\end{aligned}
$$
Our construction gives
$$
\begin{aligned}
\cM\Lambda^1(T) &= Q_{0,1,1}\times Q_{1,0,1}\times Q_{1,1,0}, \\
\cM\Lambda^2(T) &= Q_{1,0,0}\times Q_{0,1,0}\times Q_{0,0,1},
\end{aligned}
$$
where $Q_{I,J,K} = span\{x^iy^jz^k,\, 0\le i\le I,\, 0\le j\le J,
\,0\le k\le K\}$. These are identical to the lowest order
N\'{e}d\'{e}lec element defined in \cite{Nedelec80}.

Through the calculation, we also notice that on a rectangular box,
the spaces $\cM\Lambda^1(T)$ and $\cM\Lambda^2(T)$ are much smaller
than the spaces $\mathcal{W}\Lambda^1(T)$ and
$\mathcal{W}\Lambda^2(T)$ constructed in \cite{Gillette14}. For
example, one can easily see that $W_{12}\in \mathcal{W}\Lambda^1(T)$
but not in $\cM\Lambda^1(T)$. This indicates that there do exist
redundant components in $\mathcal{W}\Lambda^1(T)$ and
$\mathcal{W}\Lambda^2(T)$.

\item On a pyramid our construction is identical to the Whitney
elements constructed by Gr\v{a}dinaru and Hiptmair in
\cite{Hiptmair99}, if starting from the same $\cM\Lambda^0(T)$ as in
\cite{Hiptmair99}. Since $\cE_I=\emptyset$, one can use the
simplification given in Remark \ref{rem:noInternalEdge}, which
coincides with the construction process in \cite{Hiptmair99}. Thus
we omit the details here.
\item On a triangular prism with base defined by $(0,0)$, $(1,0)$, $(0,1)$
and the vertical limits $0<z<1$, we use the following barycentric
coordinates:
$$
\begin{aligned}
\lambda_1 &= (1-x-y)(1-z), \qquad &\lambda_2 &= x(1-z),\qquad
&\lambda_3 &= y(1-z), \\
\lambda_4 &= (1-x-y)z, \qquad &\lambda_5 &= xz,\qquad &\lambda_6 &=
yz.
\end{aligned}
$$
Our construction gives
$$
\begin{aligned}
\cM\Lambda^1(T) &= \left\{ \begin{bmatrix} (a_1-a_3y) + (a_4-a_6y)z
\\(a_2+a_3x)+(a_5+a_6x)z\\a_7+a_8x+a_9z\end{bmatrix},\, a_i\in\bbR\textrm{ for }1\le i\le 9 \right\}, \\
\cM\Lambda^2(T) &= \left\{ \begin{bmatrix} a_1x+a_2\\a_1y+a_3\\
a_4z+a_5 \end{bmatrix},\, a_i\in\bbR\textrm{ for }1\le i\le 5
\right\},
\end{aligned}
$$
which is identical to the lowest order elements on triangular prism
constructed by N\'{e}d\'{e}lec in \cite{Nedelec86}.

\item Consider a regular octahedron,
with vertices $\vv_1:\:(0,0,-1)$, $\vv_2:\:(1,0,0)$,
$\vv_3:\:(0,1,0)$, $\vv_4:\: (-1,0,0)$, $\vv_5:\:, (0,-1,0)$ and
$\vv_6:\:(0,0,1)$. The analytical form of basis functions would be
to complicated to be enclosed in this paper, or to be analyzed
directly. Here we draw the graph of two basis functions for $\cM
\Lambda^2(T)$ in Figure \ref{fig:octaBasis}. In Matlab, we are also
able to show that $\bbR^3 \subset \cM\Lambda^2(T)$ by computing
certain linear combinations of the basis functions on a fine enough
point grid, that reproduces constant vectors $[1,0,0]^t$,
$[0,1,0]^t$ and $[0,0,1]^t$ on all grid points. This numerically
verifies that $\bbR^3 \subset \cM\Lambda^2(T)$, which agrees with
the theoretical result.

\begin{figure}[h]
\begin{center}
\includegraphics[width=5cm]{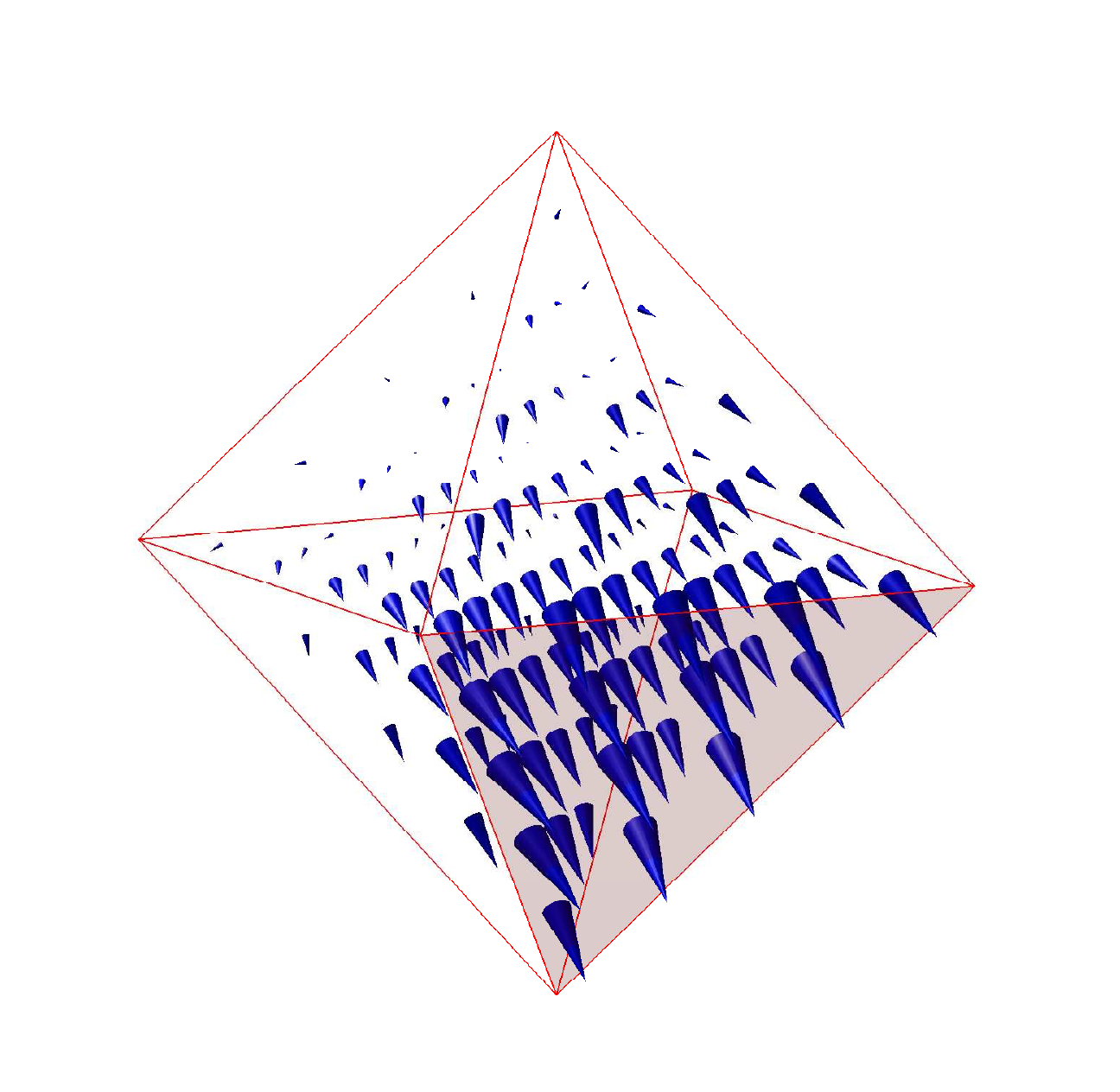}
\includegraphics[width=5cm]{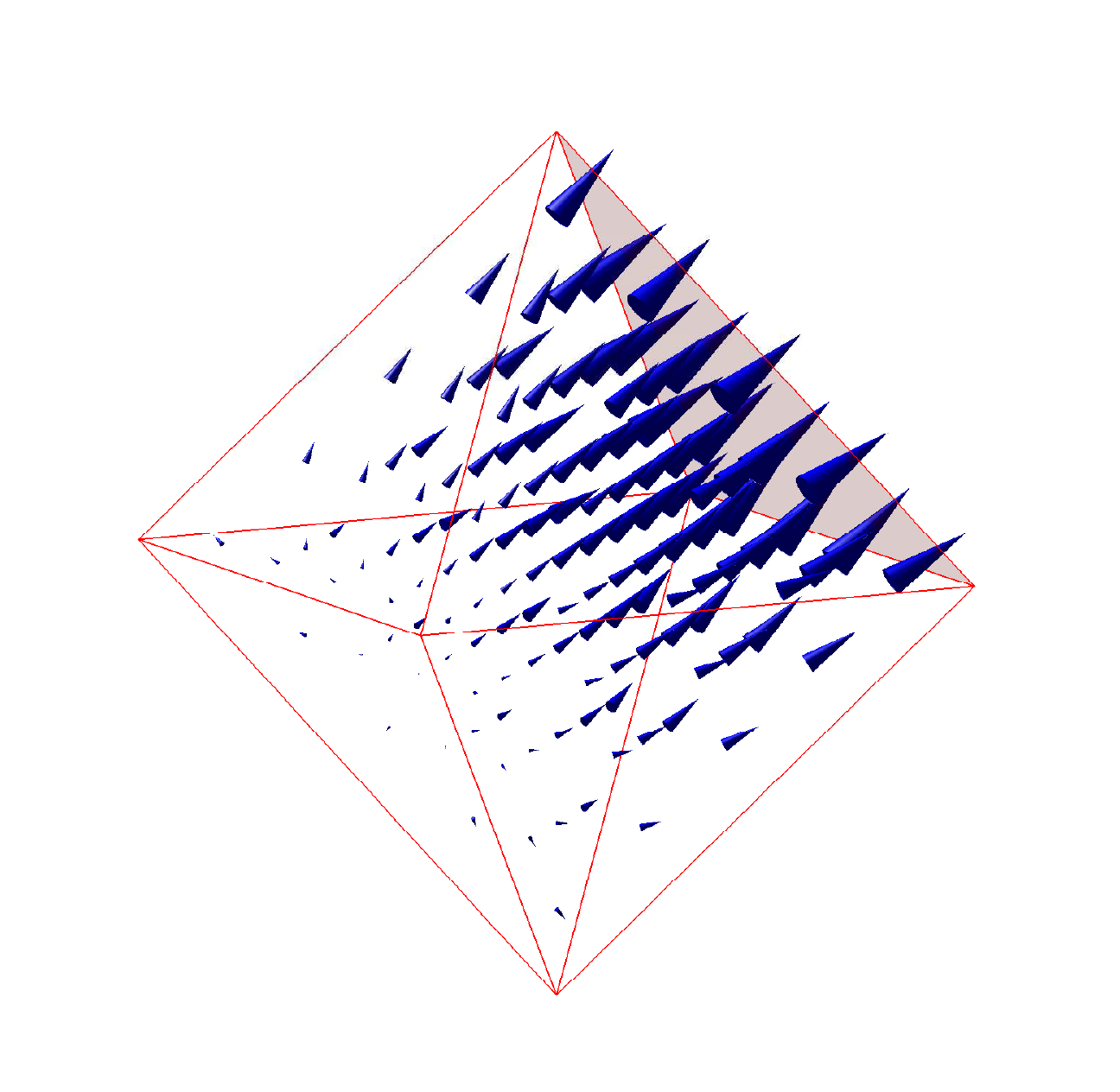}
\caption{Two basis functions for $\cM \Lambda^2(T)$ on the regular
octahedron. The normal component of the basis function is equal to
$1$ on the shaded face and $0$ on all other faces.}
\label{fig:octaBasis}
\end{center}
\end{figure}

\end{enumerate}

We end this section with a brief discussion of elements on general
hexahedra. Similar to the 2D quadrilateral case, the lowest order
Raviart-Thomas element can be defined on hexahedra via Piola
transform associated to a trilinear isomorphism, but requires
asymptotically parallelepiped grid \cite{Bermudez05} in order to
have good approximation rate. More results on the general hexahedral
N\'{e}d\'{e}lec-Raviart-Thomas elements can be found in the recent
work \cite{Falk11} and references therein. In 3D, it is also
possible for the image of the cube under a trilinear isomorphism to
have non-planar faces. By working on the physical hexahedra
directly, we can avoid this problem completely. However, a general
hexahedron does not satisfy Assumption 2, and does not belong to
either Type I or II. Nevertheless, a quick examination shows that
$\widetilde{\cM\Lambda^1}(T)$ from Remark
\ref{rem:alternativeConstruction} is still well-defined. Similar to
the proof of Lemma \ref{lem:tWcontainP1-} but requiring a more
subtle treatment on $e_{kl}\in \cE_F$, one can still show that
$\widetilde{\cM\Lambda^1}(T)$ contains $\cP_1^-\Lambda^1(T)$ and
consequently its curl contains $\bbR^3$. Therefore,
$\widetilde{\cM\Lambda^1}(T)$ may be used to build $H(\curl)$
conforming finite element spaces on hexahedral meshes. In contrast,
the situation for $\cM\Lambda^2(T)$ is much more complicated, as we
may not be able to keep the normal components on faces to be
constants. Hence it remains a topic for future research.

\appendix
\section{Adjacency matrices of a convex polyhedron}
\label{sec:adjacencyMatrices}

For a convex polyhedron $T$, we introduce a few integer-valued
matrices related to the shape of the polyhedron. For convenience,
let us temporarily index the edges in $\cE^+$ by $e_j$, for $1\le
j\le \#E$, and the faces in $\cF$ by $f_k$, for $1\le k\le \#F$.
Such kind of edge and face indices are only used in this section. In
other parts of the paper, we do not index edges or faces of a
polyhedron $T$ by a single integer, in order not to be confused with
the integer indices for vertices.

Define matrices
$$
\begin{aligned}
&A^{FtoE}:\:\bbR^{\#F} \to \bbR^{\#E},\quad&&\textrm{such that }
A^{FtoE}_{ij} = \begin{cases} 1\quad&\textrm{if }e_{i}\in\partial f_j\\ -1\quad&\textrm{if }e_{i}\in -\partial f_j\\
0\quad&\textrm{otherwise}
\end{cases}, \\
&A^{VtoE}:\:\bbR^{\#V} \to \bbR^{\#E},\quad&&\textrm{such that }
A^{VtoE}_{ij} = \begin{cases} -1\quad&\textrm{if }e_{i}\textrm{ starts from } \vv_j\\ 1\quad&\textrm{if }e_{i}\textrm{ ends at }\vv_j\\
0\quad&\textrm{otherwise}
\end{cases}. \\
\end{aligned}
$$

For each face $f_i\in \cF$, denote by $n(f_i)$ the number of edges
in $f_i$. For each $\vv_i\in\cV$, denote by $n(\vv_i)$ the number of
edges connected to $\vv_i$. Define $M^F =(A^{FtoE})^t A^{FtoE} \in
\bbR^{\#F\times\#F}$ and $M^V =(A^{VtoE})^t A^{VtoE} \in
\bbR^{\#V\times\#V}$. It is not hard to see that the entries of
$M^F$ and $M^V$ are
$$
M^F_{ij} = \begin{cases} n(f_i)\quad&\textrm{if } i=j, \\
-1 \quad&\textrm{if }f_i,\, f_j\textrm{ share an edge}, \\
0\quad&\textrm{otherwise},
\end{cases}
$$
and
$$
M^V_{ij} = \begin{cases} n(\vv_i)\quad&\textrm{if } i=j, \\
-1 \quad&\textrm{if }\vv_i,\, \vv_j\textrm{ are connected by an edge}, \\
0\quad&\textrm{otherwise}.
\end{cases}
$$

To study the rank of $M^F$ and $M^V$, let us first state a
well-known result:
\begin{lem} \label{lem:matrixRank}
Let $M$ be an irreducible and (weakly) diagonally dominant square
matrix, then $M$ either has full rank or a rank $1$ deficiency.
\end{lem}
\begin{proof}
For reader's convenience, we provide a brief proof below. A square
matrix is called irreducibly diagonally dominant if it is
irreducible, weakly diagonally dominant but in at least one row is
strictly diagonally dominant. Irreducibly diagonally dominant
matrices are non-singular. Now, by changing only one entry in any
chosen row of $M$, we can make it irreducibly diagonally dominant.
Since changing one row of a matrix can at most modify its rank by
$1$, therefore $M$ must either have full rank or a rank $1$
deficiency.
\end{proof}

Then, we have
\begin{lem} \label{lem:MFrank}
Matrix $M^F$ has rank $(\#F-1)$, and $Ker(M^F) =
span\{[1,1\ldots,1]^t\}$.
\end{lem}
\begin{proof}
By using the adjacency graph of the faces of $T$, it is not hard to
see that $M^F$ is irreducible. Since the number of faces adjacent to
each given face $f_i$ is equal to $n(f_i)$, we know that $M^F$ is
weakly diagonally dominant. By Lemma \ref{lem:matrixRank}, $M^F$
either has full rank or a rank $1$ deficiency. Indeed, $M$ has a
rank $1$ deficiency, since one can explicitly compute that
$[1,1\ldots,1]^t\in Ker(M^F)$. This completes the proof of the
lemma.
\end{proof}

\begin{lem} \label{lem:MVrank}
Matrix $M^V$ has rank $(\#V-1)$, and $Ker(M^V) =
span\{[1,1\ldots,1]^t\}$.
\end{lem}
\begin{proof}
The proof is similar to the proof of Lemma \ref{lem:MFrank}.
\end{proof}

Finally, we mention another important property of the adjacency
matrices:
\begin{lem} \label{lem:adjacencyOrthognality}
It holds that
\begin{equation} \label{eq:adjacencyOrthognality}
(A^{FtoE})^t A^{VtoE} = \mathbf{0}\quad\textrm{and}\quad
(A^{VtoE})^t A^{FtoE} = \mathbf{0}.
\end{equation}
Indeed, we have
$$
Ker((A^{FtoE})^t) = range(A^{VtoE})\quad\textrm{and}\quad
Ker((A^{VtoE})^t) = range(A^{FtoE}).
$$
\end{lem}
\begin{proof}
By using the adjacency relations, it is elementary to prove
(\ref{eq:adjacencyOrthognality}). Consequently, one has
$$
range(A^{VtoE})\subseteq Ker((A^{FtoE})^t) \quad\textrm{and}\quad
range(A^{FtoE})\subseteq Ker((A^{VtoE})^t).
$$
Now, by lemmas \ref{lem:MFrank}-\ref{lem:MVrank}, we have
$rank(A^{VtoE}) = \#V-1$ and $rank(A^{FtoE}) = \#F-1$. The lemma
follows immediately from using the rank-nullity theorem and counting
the dimensions.
\end{proof}
\section{Proof of lemmas \ref{lem:Hdivbasis3D} and \ref{lem:Hcurlbasis3D}} \label{appendix}
To prove Lemma \ref{lem:Hdivbasis3D}, we first denote
$$
\vq_f = c_{f,0} (\vx-\vx_*) + \sum_{\tilde{f}\in \cF}
c_{f,\tilde{f}} \curl\,\tW_{\tilde{f}},
$$
and then show that there exists $\{c_{f,0},\, c_{f,\tilde{f}},\,
\textrm{for } \tilde{f}\in\cF\}$ such that $\vq_f$ satisfies Lemma
\ref{lem:Hdivbasis3D}. Denote by $d_f$ the distance from $\vx_*$ to
face $f$, and by $|T_f| = \frac{1}{3} d_f|f|$ the volume of the
pyramid with base $f$ and apex $\vx_*$. For convenience, denote
$$
\delta_{f,f'} = \begin{cases} 1\quad &\textrm{if }f=f', \\
0&\textrm{otherwise}.
\end{cases}
$$
For each $f' \in \cF$, denote by $\cF(f')$ the set of all faces in
$\cF$ that share an edge with $f'$. Clearly, the number of faces in
$\cF(f')$ is equal to the number of edges of polygon $f'$, which is
denote by $n(f')$. Then, on each $f'\in\cF$, we want $\{c_{f,0},\,
c_{f,\tilde{f}},\, \textrm{for } \tilde{f}\in\cF\}$ to satisfy
\begin{equation} \label{eq:Hdivbasis3D1}
\begin{aligned}
\delta_{f,f'} = \vq_f\cdot \vn_{f'}|_{f'} &= c_{f,0}
(\vx-\vx_*)\cdot \vn_{f'}|_{f'} + \sum_{\tilde{f}\in \cF}
c_{f,\tilde{f}}
\curl\,\tW_{\tilde{f}} \cdot \vn_{f'}|_{f'} \\
&= c_{f,0} d_{f'} + c_{f,f'} \frac{n(f')}{|f'|} - \sum_{\tilde{f}\in
\cF(f')} \frac{c_{f, \tilde{f}}}{|f'|},
\end{aligned}
\end{equation}
where in the last step we have used equations
(\ref{eq:curltWf1})-(\ref{eq:curltWf2}). Multiplying both sides of
(\ref{eq:Hdivbasis3D1}) by $|f'|$ and sum up over all $f'\in\cF$
gives
$$
|f| = \sum_{f'\in \cF} c_{f,0} d_{f'} |f'| + 0 = 3 c_{f,0} |T|,
$$
which implies
$$
c_{f,0} = \frac{|f|}{3|T|}.
$$
Now, Equation (\ref{eq:Hdivbasis3D1}) can be rewritten into, for
each $f'\in\cF$,
$$
n(f') c_{f,f'} - \sum_{\tilde{f}\in \cF(f')}c_{f, \tilde{f}} =
\delta_{f,f'} |f'| - \frac{|T_{f'}|}{|T|} |f|.
$$
This provides a linear system for solving $c_{f,\tilde{f}},\,
\textrm{for all } \tilde{f}\in\cF$, where the coefficient matrix is
exactly $M^F$ defined in Appendix \ref{sec:adjacencyMatrices}. Note
the right-hand side of the above linear system is obviously
orthogonal to $Ker(M^F)$, as
$$
\sum_{f'\in\cF} \left(\delta_{f,f'} |f'| - \frac{|T_{f'}|}{|T|}|f|
\right) = 0.
$$
Therefore the linear system is solvable. This establishes the
existence of $\vq_f$ satisfying $\vq_f\cdot\vn_{f'}|_{f'} =
\delta_{f,f'}$. From the construction we also know that $\vq_f$ is
computable, with details given at the end of this section. Moreover,
$\vq_f$ is indeed uniquely defined since by setting
$c_{f,\tilde{f}}=1$ for all $\tilde{f}\in \cF$, i.e., by making the
coefficients in $Ker(M^F)$, one would get
$$
\sum_{\tilde{f}\in \cF} c_{f,\tilde{f}} \curl\,\tW_{\tilde{f}} =
\curl \sum_{\tilde{f}\in \cF} \tW_{\tilde{f}} = \curl \mathbf{0} =
\mathbf{0},
$$
where we have used the simple fact that $\sum_{f\in \cF} \tW_f =
\mathbf{0}$ according to the definition of $\tW_f$.

It is not hard to see that $\{\vq_f,\,\textrm{for }f\in\cF\}$ is
linearly independent. Again, by using $\sum_{f\in \cF} \tW_f =
\mathbf{0}$, we have
$$
\begin{aligned}
\dim \cM\Lambda^2(T) & \le \dim \curl\left( span\{\tW_f,\,
f\in\cF\}\right) +
\dim span\{\vx - \vx_*\} \\
&\le \dim span\{\tW_f,\, f\in\cF\} + 1 \\
&\le (\#F-1) + 1 = \#F.
\end{aligned}
$$
Combining the above, $\{\vq_f,\,\textrm{for }f\in\cF\}$ must form a
basis for $\cM\Lambda^2(T)$ and consequently $\dim \cM\Lambda^2(T) =
\#F$. This completes the proof of Lemma \ref{lem:Hdivbasis3D}.

Next we prove Lemma \ref{lem:Hcurlbasis3D}. The idea is similar to
the proof of Lemma \ref{lem:Hdivbasis3D}. We express
$$ \vp_e = \sum_{i=1}^n a_{e,i} \nabla\lambda_i + \sum_{f\in\cF}
b_{e,f} \tW_f.
$$
Now, let $e'\in \cE^+$. Denote by $\vv_{\alpha}$ and $\vv_{\beta}$
the starting and ending vertices of $e'$, and by $f_l$/$f_r$ the
faces to the left/right of edge $e'$, seeing from outside of $T$.
Then, by Assumption 1, Lemma \ref{lem:twijOnEdges} and the
definition of $\tW_f$, one has
$$
\begin{aligned}
\delta_{e,e'} &= \vp_e\cdot\vt_{e'}|_{e'} = \sum_{i=1}^n a_{e,i}
\nabla\lambda_i \cdot\vt_{e'}|_{e'} + \sum_{f\in\cF} b_{e,f} \tW_f
\cdot\vt_{e'}|_{e'} \\
&= \frac{-a_{e,\alpha} + a_{e,\beta}}{|e'|} + \frac{b_{e,f_l} -
b_{e,f_r}}{|e'|},
\end{aligned}
$$
which we further rewrite into
\begin{equation} \label{eq:matrixA}
-a_{e,\alpha} + a_{e,\beta} + b_{e,f_l} - b_{e,f_r} = \delta_{e,e'}
|e'|.
\end{equation}
The above equation holds on every $e'\in\cE^+$, and thus gives us a
linear system with $\#E$ equations and $\#V+\#F = \#E+2$ unknowns.
Denote by $A:\: \bbR^{\#E+2} \to \bbR^{\#E}$ the coefficient matrix
of this linear system. It is not hard to see that, under proper
ordering, one has
$$
A = [A^{VtoE} \; A^{FtoE}],
$$
where $A^{VtoE}$ and $A^{FtoE}$ are as defined in Appendix
\ref{sec:adjacencyMatrices}.

By Lemma \ref{lem:adjacencyOrthognality}, we have
$$
A^tA = \begin{bmatrix}M^V & \mathbf{0} \\ \mathbf{0} &
 M^F \end{bmatrix} \in\bbR^{(\#E+2)\times(\#E+2)}.
$$
Consequently, by lemmas \ref{lem:MFrank}-\ref{lem:MVrank}, we know
that $rank(A) = rank(A^tA) = (\#V-1)+ (\#F-1) = \#E$ and $Ker (A)$
is spanned by the following two vectors:
\begin{equation} \label{eq:KerM}
\begin{aligned}
&[1,1,\ldots,1,0,0,\ldots,0]^t, \quad \textrm{with }(\#V)\, 1's \textrm{ and }(\#F)\, 0's, \\
&[0,0,\ldots,0,1,1,\ldots,1]^t, \quad \textrm{with }(\#V)\, 0's
\textrm{ and }(\#F)\, 1's.
\end{aligned}
\end{equation}
Then, the linear system (\ref{eq:matrixA}) is solvable. Moreover, we
realize that $\vp_e$ is indeed uniquely defined, as all coefficients
in $Ker (A)$ only generate zero functions because
$\sum_{i=1}^n\nabla\lambda_i = \mathbf{0}$ and $\sum_{f\in \cF}
\tW_f = \mathbf{0}$.

We can similarly show that $\{\vp_e,\,\textrm{for }e\in\cE^+\}$ is
linearly independent, and thus by counting dimensions, it form a
basis for $\cM\Lambda^1(T)$. This completes the proof of Lemma
\ref{lem:Hcurlbasis3D}.

Finally, we briefly discuss how to compute the basis functions in
practice. Using elementary linear algebra, it is not hard to see
that:
\begin{enumerate}
\item To compute $\vq_f$, one needs to solve a linear system $M^F\vu =
\vb$, where $M^F \in \bbR^{\#F\times\#F}$ has a non-trivial kernel
containing all constant vectors, and $\vb\in Ker(M^F)^{\perp} =
Range(M^F)$. Indeed, solving $M^F \vu = \vb$ is equivalent to
solving a non-singular square system
$$
\begin{bmatrix} M^F & \mathbf{1} \\ \mathbf{1}^t & 0\end{bmatrix}
\begin{bmatrix}\vu \\ 0\end{bmatrix} = \begin{bmatrix}\vb \\
0\end{bmatrix},
$$
where $\mathbf{1}$ denote a constant column vector with all entries
equal to $1$.
\item To compute $\vq_e$, one needs to solve a linear system $A\vu =
\vb$, where $A = [A^{VtoE}\, A^{FtoE}] \in \bbR^{\#E\times(\#E+2)}$
has rank $\#E$ and kernel spanned by vectors in (\ref{eq:KerM}).
Indeed, solving $A\vu = \vb$ is equivalent to solving  a
non-singular square system
$$
\begin{bmatrix} A^{VtoE} & A^{FtoE} \\ \mathbf{1}^t & \mathbf{0}^t \\ \mathbf{0}^t & \mathbf{1}^t \end{bmatrix}
\vu  = \begin{bmatrix}\vb \\
0 \\0 \end{bmatrix}.
$$
\end{enumerate}

\section{Proof of Lemma \ref{lem:ContainingR}} \label{appendix2}

For Type I polyhedra, the proof is easy. By Remark
\ref{rem:noInternalEdge}, we have $\vp_{e_{ij}}=|e_{ij}|\tW_{ij}$ on
each $e_{ij}\in\cE^+$. Thus by Lemma \ref{lem:tWcontainP1-}, one
immediately gets $\cP_1^-\Lambda^1(T) \subseteq \cM\Lambda^1(T)$.
This, together with the fact that $\curl (\va\times \vx) = 2\va$ for
all $\va\in\bbR^3$, implies that $\bbR^3\subset \cM\Lambda^2(T)$.
Finally, since $span\{\vx-\vx_*\} \subset \cM\Lambda^2(T)$, we have
$\cP_1^-\Lambda^2(T)\subseteq \cM\Lambda^2(T)$.

Now let us consider Type II polyhedra. From the proof of Lemma
\ref{lem:tWcontainP1-}, one has
$$
2\sum_{e_{ij}\in\cE^+} ((\va\times (\vv_i-\vx_c)) \cdot\vtau_{ij})
\tW_{ij} = \sum_{e_{ij}\in\cE} ((\va\times (\vv_i-\vx_c))
\cdot\vtau_{ij}) \tW_{ij} = 2\va\times (\vx-\vx_c),
$$
for all $\va\in\bbR^3$.  If we can show that
\begin{equation} \label{eq:Appb1}
\sum_{e_{ij}\in\cE^+} ((\va\times (\vv_i-\vx_c)) \cdot\vtau_{ij})
\tW_{ij} = \sum_{f\in\cF} C_f \tW_f \in \cM\Lambda^1(T),
\end{equation}
this together with the face that $\bbR^3\subset \nabla
\cM\Lambda^0(T) \subset \cM\Lambda^1(T)$ will imply
$\cP_1^-\Lambda^1(T)\subseteq \cM\Lambda^1(T)$. And consequently one
will be able to prove that $\cP_1^-\Lambda^2(T)\subseteq
\cM\Lambda^2(T)$. Next, we focus on prove the existence of a set of
coefficients $\{C_f,\textrm{ for }f\in\cF\}$ that satisfies
(\ref{eq:Appb1}).

For each $e_{ij}\in\cE^+$, denote by $f^{ij}_l$ and $f^{ij}_r$ the
faces on the left and right side of $e_{ij}$ respectively, seeing
from outside of $T$. Notice that the right-hand side of Equation
(\ref{eq:Appb1}) can further be written into $\sum_{e_{ij}\in\cE^+}
(C_{f^{ij}_l} - C_{f^{ij}_r}) \tW_{ij}$. Thus it remains to prove
that the system
\begin{equation} \label{eq:Appb2}
C_{f^{ij}_l} - C_{f^{ij}_r} = (\va\times (\vv_i-\vx_c))
\cdot\vtau_{ij},\qquad\textrm{for all }e_{ij}\in\cE^+,
\end{equation}
is solvable. The coefficient matrix of system (\ref{eq:Appb2}) is
exactly $A^{FtoE}$, as defined in Appendix
\ref{sec:adjacencyMatrices}. Denote the right-hand side vector of
system (\ref{eq:Appb2}) by
$$\vb = [(\va\times (\vv_i-\vx_c))
\cdot\vtau_{ij}]_{e_{ij}\in\cE^+} \in \bbR^{\#E}.
$$
The linear system (\ref{eq:Appb2}) is solvable only if
$$
\vb \in Range(A^{FtoE}) = Ker((A^{FtoE})^t)^\perp.
$$
By Lemma \ref{lem:adjacencyOrthognality}, we have
$Ker((A^{FtoE})^t)^\perp = range(A^{VtoE})^{\perp} =
Ker((A^{VtoE})^t)$. Therefore, System (\ref{eq:Appb2}) is solvable
as long as $(A^{VtoE})^t \vb = \mathbf{0}$, which can be explicitly
written as
\begin{equation} \label{eq:Appb3}
\sum_{j,\,e_{ij}\in\cE^+} b_{ij} - \sum_{j,\,e_{ji}\in\cE^+} b_{ji}
= 0,\qquad\textrm{for all }1\le i\le n,
\end{equation}
where we conveniently denote by $b_{ij}$ the entry of vector $\vb$
corresponding to $e_{ij}\in\cE^+$. According to (\ref{eq:Appb2}),
$b_{ij}$ can be viewed as the jump of coefficient $C_f$ across the
edge $e_{ij}$. Thus the constraints given by (\ref{eq:Appb3}) are
equivalent to say that, the summation of such jumps over all edges
connecting to one given vertex should be $0$. By the definition of
$\vb$ and the fact that $\vv_i\times\vtau_{ij} =
\vv_j\times\vtau_{ij}$, Equation (\ref{eq:Appb3}) is equivalent to
$$
\sum_{j,\,e_{ij}\in\cE^+} (\va\times (\vv_i-\vx_c)) \cdot\vtau_{ij}
 - \sum_{j,\,e_{ji}\in\cE^+} (\va\times (\vv_j-\vx_c)) \cdot\vtau_{ji}
  = \sum_{j,\,e_{ij}\in\cE} (\va\times (\vv_i-\vx_c))
  \cdot\vtau_{ij} = 0,
$$
for all $1\le i\le n$, which is true on Type II polyhedra. In other
words, we have shown that for Type II polyhedra, Equation
(\ref{eq:Appb2}) is solvable. This completes the proof of Lemma
\ref{lem:ContainingR}.

\bibliographystyle{amsplain}

\end{document}